\documentclass{amsart}
\usepackage{amssymb}
\usepackage{amsmath}
\usepackage{amsthm}
\usepackage{mathrsfs}
\usepackage{colortbl}
\usepackage{graphicx}
\newtheorem{theorem}{Theorem}
\newtheorem{lemma}[theorem]{Lemma}

\newtheorem{remark}[theorem]{Remark}

\theoremstyle{definition}
\newtheorem{definition}[theorem]{Definition}

\newtheorem{proposition}[theorem]{Proposition}

\makeatletter

\@addtoreset{theorem}{section}
\makeatother

\makeatletter

\@addtoreset{equation}{section}
\makeatother

\begin{document}                                                 

\title[Heat equations in two half spaces and an interface]{Global solvability for the heat equations in two half spaces and an interface}                                
\author[Hajime Koba]{Hajime Koba}                                
\address{Graduate School of Engineering Science, the University of Osaka\\
1-3 Machikaneyamacho, Toyonaka, Osaka, 560-8531, Japan}                                  
\email{koba.hajime.es@osaka-u.ac.jp}


\begin{abstract}
This paper considers the existence of a global-in-time strong solution to the heat equations in the two half spaces $\mathbb{R}^3_+(=\mathbb{R}^2 \times (0,\infty))$, $\mathbb{R}^3_-(= \mathbb{R}^2 \times (-\infty ,0))$, and the interface $\mathbb{R}^2 \times \{ 0 \} (\cong \mathbb{R}^2)$. We introduce and study some function spaces in the two half spaces and the interface. We apply our function spaces and the maximal $L^p$-regularity for Hilbert space-valued functions to show the existence of a local-in-time strong solution to our heat equations. By using an energy equality of our heat system, we prove the existence of a unique global-in-time strong solution to the system with large initial data. The key idea of constructing strong solutions to our system is to make use of nice properties of the heat semigroups and kernels for $\mathbb{R}^3_+$, $\mathbb{R}^3_{-}$, and $\mathbb{R}^2$. In Appendix, we derive our heat equations in the two half spaces and the interface from an energetic point of view.
\end{abstract}
\maketitle

\section{Introduction}

\begin{figure}[htbp]
\includegraphics[width=12cm]{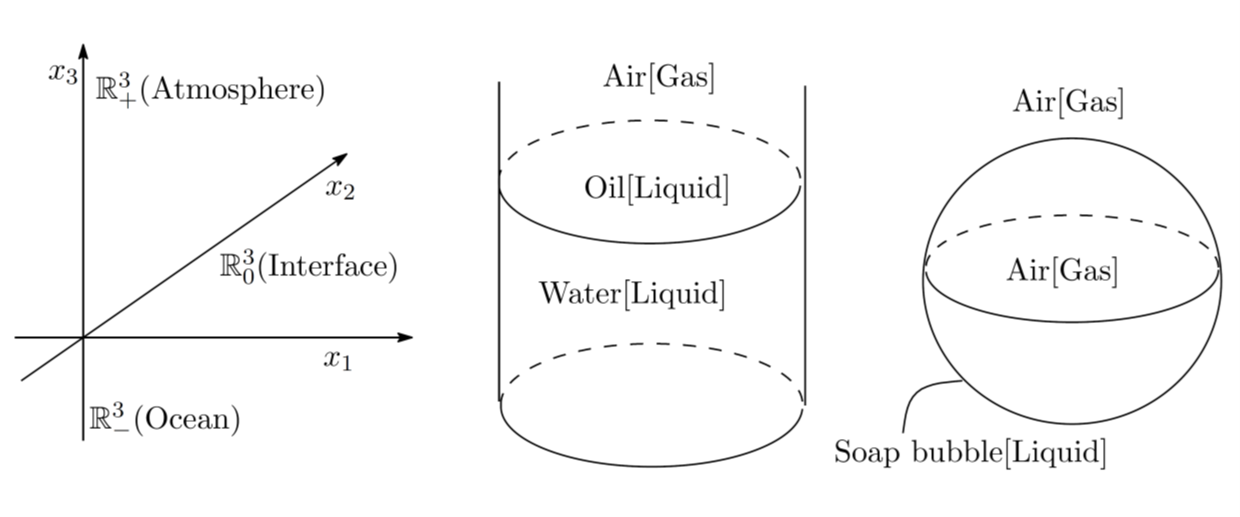}
\caption{Three-phase problems}
\label{Fig1}
\end{figure}

We are interested in the existence of a global-in-time strong solution to the heat equations in the two half spaces $\mathbb{R}^3_+$, $\mathbb{R}^3_-$, and the interface $\mathbb{R}^2 \times \{ 0 \} (\cong \mathbb{R}^2)$. The heat system is a simple heat transfer model for ocean-atmosphere with an interface, oil floating on water, or an soap bubble flying in the air (see Figure \ref{Fig1}). This paper has two purposes. The first one is to introduce and study some function spaces in the two half spaces and the interface. The second one is to apply our function spaces and the maximal $L^p$-regularity for Hilbert space-valued functions to construct a strong solution of our heat equations.

Let us first introduce basic notations. Let $x = { }^t (x_1, x_2, x_3) \in \mathbb{R}^3$, $x_h = { }^t (x_1, x_2) \in \mathbb{R}^2$ be the spatial variables, and $t , \tau, t_1 ,t_2 \geq 0$ be the time variables. The symbols $\nabla$, $\nabla_h$, and $\Delta$, $\Delta_h$ are two gradient and two Laplace operators defined by $\nabla = { }^t (\partial_1 , \partial_2 , \partial_3)$, $\nabla_h = { }^t ( \partial_1 , \partial_2)$, $\Delta = \partial_1^2 + \partial_2^2 + \partial_3^2$, and $\Delta_h = \partial_1^2 + \partial_2^2$, where $\partial_j = \partial/{\partial x_j}$. Set 
\begin{equation*}
\mathbb{R}^3_{+} := \{ x \in \mathbb{R}^3; { }x_3 >0 \},{ \ } \mathbb{R}^3_{-} := \{ x \in \mathbb{R}^3; { }x_3 < 0 \},{ \ } \mathbb{R}^3_{0} := \{ x \in \mathbb{R}^3; { }x_3 = 0 \}.
\end{equation*}
In this paper we equate functions in $\mathbb{R}^3_{0}$ to functions in $\mathbb{R}^2$.

This paper considers the existence of a global-in-time strong solution to the heat equations in the two half spaces $\mathbb{R}^3_+$, $\mathbb{R}^3_-$, and the interface $\mathbb{R}^2$:
\begin{equation}\label{eq11}
\begin{cases}
\alpha_A \partial_t \theta_A = \kappa_A \Delta \theta_A & \text{ in } \mathbb{R}^3_+ \times (0, T) ,\\
\alpha_B \partial_t \theta_B = \kappa_B \Delta \theta_B & \text{ in } \mathbb{R}^3_{-} \times (0, T ),\\
\alpha_S \partial_t \theta_S = \kappa_S \Delta_h \theta_S + \kappa_A \gamma_{+}[ \partial_3 \theta_A] - \kappa_B \gamma_{-}[ \partial_3 \theta_B] & \text{ in } \mathbb{R}^2 \times (0, T ),\\
\gamma_{+} [\theta_A] = \gamma_{-} [\theta_B] = \theta_S & \text{ in } \mathbb{R}^2 \times (0, T)
\end{cases}
\end{equation}
with
\begin{equation*}
\begin{cases}
{\displaystyle{\lim_{ \vert x \vert \to \infty , { \ }x \in \mathbb{R}^3_+ } \theta_A =0}},{ \ }{\displaystyle{\lim_{ \vert x \vert \to \infty,{ \ }x \in \mathbb{R}^3_{-} } \theta_B  =0}},{ \ }{\displaystyle{\lim_{ \vert x_h \vert \to \infty,{ \ }x_h \in \mathbb{R}^2 } \theta_S =0}} & \text{ on } (0, T),\\
\theta_A \vert_{t=0} = \theta^A_0{ \ } \text{ in } \mathbb{R}^3_+, { \ }\theta_B \vert_{t = 0} = \theta_0^B{ \ } \text{ in }\mathbb{R}^3_{-}, { \ }\theta_S \vert_{t=0} = \theta_0^S{ \ } \text{ in } \mathbb{R}^2,
\end{cases}
\end{equation*}
where $T \in (0, \infty]$, and $\gamma_{+}$, $\gamma_{-}$ are two trace operators such that
\begin{align*}
\gamma_{+}[\cdot] : W^{1,2} (\mathbb{R}^3_{+}) \to L^2 (\mathbb{R}^2)(= L^2 (\partial \mathbb{R}^3_+)),\\
\gamma_{-}[\cdot] : W^{1,2} (\mathbb{R}^3_{-}) \to L^2 (\mathbb{R}^2)(= L^2(\partial \mathbb{R}^3_-)).
\end{align*}
The unknown functions $\theta_A = \theta_A (x,t)$, $\theta_B = \theta_B (x,t)$, and $\theta_S = \theta_S (x_h,t)$ are the temperatures of the fluid(or substance) in $\mathbb{R}^3_{+}$, $\mathbb{R}^3_{-}$, and $\mathbb{R}^3_{0}$, respectively. The given positive constants $\kappa_A$, $\kappa_B$, and $\kappa_S$ are the thermal conductivities of the fluid(or substance) in $\mathbb{R}^3_{+}$, $\mathbb{R}^3_{-}$, and $\mathbb{R}^3_0$, respectively. The given functions $\theta^A_0 = \theta^A_0 (x)$, $\theta^B_0 = \theta^B_0 (x)$, $\theta^S_0 = \theta^S_0 (x_h)$ are three initial data. The symbols $\alpha_A$, $\alpha_B$, and $\alpha_S$ are given positive constants. Note that system \eqref{eq11} is just one possibility of the dominant equations for the temperatures on three phase problems. One can see Appendix (I) for the derivation of system \eqref{eq11} and $(\alpha_A,\alpha_B,\alpha_S)$.

Let us explain the background of this study. We are interested in the temperatures on three-phase problems such as ocean-atmosphere with an interface, oil floating on water, or an soap bubble flying in the air. The ocean-atmosphere interface is considered to play an important role in the flow and temperature of the fluid in the atmosphere and ocean. There are two cases in the study of the interface. One is when there is a material density at the interface, and the other is when there is no material density at the interface. We often call a material density at the interface a \emph{surface mass}. Gurtin \cite{Gur93} introduced some multiphase system without surface mass. Gatignol-Prud'homme \cite{GP01} and Slattery-Sagis-Oh \cite{SSO07} introduced some multiphase system with surface mass. It is necessary to consider the physical phenomena that occur at the interface when we study either surface tension or the movement of an interface. Indeed, Koba \cite{K23,K25} admitted the existence of surface mass in order to derive surface tension from a theoretical point of view. In this paper, we admit the existence of surface mass at the interface $\mathbb{R}^2 \times \{ 0 \}$.

To construct a strong solution of system \eqref{eq11}, we have to deal with both
\begin{equation}\label{eq12}
\kappa_A \gamma_{+}[ \partial_3 \theta_A] - \kappa_B \gamma_{-}[ \partial_3 \theta_B] \text{ and }\gamma_{+} [\theta_A] = \gamma_{-} [\theta_B] = \theta_S.
\end{equation} 
However, it is not easy to handle \eqref{eq12} directly, and the trace operators $\gamma_+,\gamma_-$ are not closed operators in general. To overcome these difficulties, we apply our function spaces in $\mathbb{R}^3_+$, $\mathbb{R}^3_-$, $\mathbb{R}^2$, and the maximal $L^p$-regularity for Hilbert space-valued functions. For that reason, in this paper we assume that
\begin{equation*}
\alpha_A = 1, { \ }\alpha_B = 1,{ \ } \kappa_S = \tilde{\kappa}_S \alpha_S.
\end{equation*}
We make use of both our maximal $L^p$-regularity theory and the condition that $\kappa_S = \tilde{\kappa}_S \alpha_S$  to construct a local-in-time strong solution of our system. Our methods allow us to handle the case when $\alpha_A >0$ and $\alpha_B >0$. For readability, we assume that $\alpha_A = 1$ and $\alpha_B = 1$ in this paper.

Now we state the main results of this paper.
\begin{theorem}\label{thm11}
Let $\kappa_A , \kappa_B , \tilde{\kappa}_S, \alpha_S >0$. Then there is $\alpha_0 = \alpha_0 ( \tilde{\kappa}_S) >0$ such that if $\alpha_S > \alpha_0$ then for each $\theta_0^A \in W^{1,2} (\mathbb{R}^3_{+})$, $\theta_0^B \in W^{1,2}( \mathbb{R}^3_{-})$, $\theta_0^S \in W^{1,2} (\mathbb{R}^2)$ satisfying $\gamma_{+}[\theta_0^A] = \theta_0^S$ and $\gamma_{-}[\theta_0^B] = \theta_0^S$, system \eqref{eq11} admits a unique global-in-time strong solution $(\theta_A , \theta_B, \theta_S)$:
\footnotesize
\begin{align*}
\theta_A & \in C ([0, \infty); L^2 (\mathbb{R}^3_{+})) \cap L_{loc}^2 (0,\infty; W^{2,2}(\mathbb{R}^3_+)) \cap W_{loc}^{1,2}(0,\infty ; L^2(\mathbb{R}^3_+)) \cap L^2_{loc}( \mathbb{R}_+ \times \mathbb{R}^3_{+}),\\
\theta_B & \in C ([0, \infty); L^2 (\mathbb{R}^3_{-})) \cap L_{loc}^2 (0,\infty; W^{2,2}(\mathbb{R}^3_{-})) \cap W_{loc}^{1,2}(0,\infty ; L^2(\mathbb{R}^3_{-})) \cap L^2_{loc}(\mathbb{R}_+ \times \mathbb{R}^3_{-}),\\
\theta_S & \in C ([0, \infty); L^2 (\mathbb{R}^2)) \cap L_{loc}^2 (0,\infty; W^{2,2}(\mathbb{R}^2)) \cap W_{loc}^{1,2}(0,\infty ; L^2(\mathbb{R}^2)) \cap L_{loc}^2 (\mathbb{R}_+ \times \mathbb{R}^2).
\end{align*}\normalsize
Moreover, for $0 \leq t_1 \leq t_2 < \infty$
\begin{multline}\label{eq13}
\Vert \theta_A (t_2) \Vert^2_{L^2 (\mathbb{R}^3_{+})} + \Vert \theta_B (t_2) \Vert^2_{L^2 (\mathbb{R}^3_{-})} + \alpha_S \Vert \theta_S (t_2) \Vert^2_{L^2 (\mathbb{R}^2)}\\
 + 2 \kappa_A \int_{t_1}^{t_2} \Vert \nabla \theta_A ( \tau) \Vert^2_{L^2 (\mathbb{R}^3_{+})}{ \ } d \tau + 2 \kappa_B \int_{t_1}^{t_2} \Vert \nabla \theta_B ( \tau ) \Vert^2_{L^2 (\mathbb{R}^3_{-})}{ \ }d \tau \\+  2 \tilde{\kappa}_S \alpha_S \int_{t_1}^{t_2} \Vert \nabla_h \theta_S (\tau) \Vert^2_{L^2 (\mathbb{R}^2)}{ \ }d \tau\\ = \Vert \theta_A (t_1) \Vert^2_{L^2 (\mathbb{R}^3_{+})} + \Vert \theta_B (t_1) \Vert^2_{L^2 (\mathbb{R}^3_{-})} + \alpha_S \Vert \theta_S (t_1) \Vert^2_{L^2 (\mathbb{R}^2)},
\end{multline}
and
\begin{equation*}
\lim_{t \to 0 + 0} ( \Vert \theta_A (t) - \theta_0^A \Vert_{L^2(\mathbb{R}^3_+)} + \Vert \theta_B (t) - \theta_0^B \Vert_{L^2(\mathbb{R}^3_-)} + \Vert \theta_S (t) - \theta_0^S \Vert_{L^2 (\mathbb{R}^2)}  ) = 0.
\end{equation*}
Here
\begin{align*}
& L^2_{loc}(0,\infty; W^{2,2} (\Omega)) = \{ \varphi; \varphi \in L^2(0,T; W^{2,2} (\Omega)) \text{ for each fixed } T >0  \},\\
& W^{1,2}_{loc}(0,\infty; L^2 (\Omega)) = \{ \varphi; \varphi \in W^{1,2}(0,T; L^2 (\Omega)) \text{ for each fixed } T >0  \},\\
& L^2_{loc}( \mathbb{R}_+ \times \Omega ) = \{ \varphi; \varphi \in L^2((0,T) \times \Omega ) \text{ for each fixed } T >0  \},
\end{align*}
where $\Omega = \mathbb{R}^3_+, \mathbb{R}^3_-$, or $\mathbb{R}^2$.
\end{theorem}

Let us explain the key ideas of constructing a strong solution to our heat equations. Let $\theta_0^A \in W^{1,2} (\mathbb{R}^3_{+})$, $\theta_0^B \in W^{1,2}( \mathbb{R}^3_{-})$, and $\theta_0^S \in W^{1,2} (\mathbb{R}^2)$ such that $\gamma_{+} [ \theta_0^A ] = \theta_0^S$ and $\gamma_{-}[\theta_0^B] = \theta_0^S$. Let $\beta >0$. Assume that $\theta_A$, $\theta_B$, and $\theta_S$ are smooth functions. Set
\begin{equation*}
\begin{cases}
u_A = u_A (x , t ) = \theta_A (x ,t ) - \theta_S ( x_h , t) {\rm{e}}^{- \beta x_3},\\
u_B = u_B (x , t ) = \theta_B (x ,t ) - \theta_S ( x_h , t) {\rm{e}}^{ \beta x_3},\\
u_S = u_S (x_h , t ) = \theta_S ( x_h , t ),
\end{cases}\begin{cases}
u_0^A = u_0^A (x) = \theta_0^A - \theta_0^S {\rm{e}}^{- \beta x_3},\\
u_0^B = u_0^B (x) = \theta_0^B  - \theta_0^S {\rm{e}}^{ \beta x_3},\\
u_0^S = u_0^S (x_h) = \theta_0^S.
\end{cases}
\end{equation*}
Then
\begin{equation}\label{eq14}
\begin{cases}
\partial_t u_A - \kappa_A \Delta u_A = F_1(u_A, u_B,u_S)  & \text{ in } \mathbb{R}^3_+ \times (0, T) ,\\
\partial_t u_B - \kappa_B \Delta u_B = F_2 (u_A, u_B, u_S) & \text{ in } \mathbb{R}^3_{-} \times (0, T ),\\
\partial_t u_S - \tilde{\kappa}_S \Delta_h u_S = F_3 (u_A, u_B , u_S) & \text{ in } \mathbb{R}^2 \times (0, T ),\\
\gamma_+ [u_A] = \gamma_{-} [u_B] = 0 & \text{ in } \mathbb{R}^2 \times (0, T),
\end{cases}
\end{equation}
with
\begin{equation*}
\begin{cases}
{\displaystyle{\lim_{ \vert x \vert \to \infty , { \ }x \in \mathbb{R}^3_+ } u_A =0}},{ \ }{\displaystyle{\lim_{ \vert x \vert \to \infty,{ \ }x \in \mathbb{R}^3_{-} } u_B  =0}},{ \ }{\displaystyle{\lim_{ \vert x_h \vert \to \infty,{ \ }x_h \in \mathbb{R}^2 } u_S =0}} & \text{ on } (0, T),\\
u_A \vert_{t=0} = u^A_0{ \ } \text{ in } \mathbb{R}^3_+, { \ }u_B \vert_{t = 0} = u_0^B{ \ } \text{ in }\mathbb{R}^3_{-}, { \ }u_S \vert_{t=0} = u_0^S{ \ } \text{ in } \mathbb{R}^2,
\end{cases}
\end{equation*}
where
\begin{align*}
F_1 (u_A ,u_B ,u_S ) &= - ( \partial_t u_S ) {\rm{e}}^{- \beta x_3} + (\kappa_A \Delta_h u_S ) {\rm{e}}^{- \beta x_3} + \beta^2 \kappa_A u_S {\rm{e}}^{- \beta x_3},\\
F_2 (u_A , u_B , u_S ) &= - ( \partial_t u_S ) {\rm{e}}^{\beta x_3} + ( \kappa_B \Delta_h u_S ) {\rm{e}}^{\beta x_3} + \beta^2 \kappa_B u_S {\rm{e}}^{ \beta x_3},\\
F_3 ( u_A , u_B ,u_S ) &= \frac{\kappa_A}{\alpha_S} \gamma_{+}[ \partial_3 u_A] - \frac{\kappa_B}{\alpha_S} \gamma_{-}[ \partial_3 u_B] - \frac{\beta}{\alpha_S} (\kappa_A + \kappa_B) u_S.
\end{align*}
Remark that
\begin{equation*}
u_0^A \in W_0^{1,2} ( \mathbb{R}^3_{+}),  { \ }u_0^B \in W_0^{1,2} ( \mathbb{R}^3_{-}), { \ }u_0^S \in W^{1,2} ( \mathbb{R}^2).
\end{equation*}

Applying a strong solution to system \eqref{eq14}, we construct a strong solution to system \eqref{eq11}. To show the existence of a strong solution to \eqref{eq14}, we introduce and study the following function spaces in $\mathbb{R}^3_+$, $\mathbb{R}^3_-$, and $\mathbb{R}^2$:
\begin{align*}
H & = \{ f = { }^t (f_A , f_B, f_S ) \in L^2 (\mathbb{R}^3_+) \times L^2 (\mathbb{R}^3_-) \times L^2 (\mathbb{R}^2); { \ } \Vert f \Vert_H < \infty \},\\
X_T & = \{ \varphi \in C ([0,T); H ); { \ } \Vert \varphi \Vert_{X_T} < \infty \}
\end{align*}
with 
\begin{align*}
\Vert f \Vert_H & = (\Vert f_A \Vert_{L^2(\mathbb{R}^3_+)}^2 + \Vert f_B \Vert_{L^2 (\mathbb{R}^3_-)}^2 + \Vert f_S \Vert_{L^2 (\mathbb{R}^2)}^2)^{1/2},\\
\Vert \varphi \Vert_{X_T} & = \sup_{0 \leq t <T}\Vert \varphi \Vert_H + \Vert d \varphi /{d t } \Vert_{L^2(0,T;H)} + \Vert L \varphi \Vert_{L^2(0,T;H)},
\end{align*}
where $L$ is the linear operator in $H$ defined by
\begin{equation*}
\begin{cases}
L f = { }^t ( - \kappa_A \Delta f_A , - \kappa_B \Delta f_B , - \tilde{\kappa}_S \Delta_h f_S ),\\
D (L) = [W_0^{1,2} \cap W^{2,2} (\mathbb{R}^3_+)] \times [W_0^{1,2} \cap W^{2,2} (\mathbb{R}^3_-)] \times W^{2,2} (\mathbb{R}^2).
\end{cases}
\end{equation*}
Using nice properties of $L$, we show the existence of a local-in-time strong solution to \eqref{eq14}. This paper studies the maximal $L^p$-regularity of $L$. Therefore, our results are extensions of some results in Desch-Hieber-Pr\"{u}ss \cite{DHP01}. They studied maximal $L^p$-regularity for the Stokes and Laplace operators in a half space. We improve a method in Koba \cite{K22} to construct a local-in-time strong solution to system \eqref{eq14}. They showed the existence of a local-in-time and global-in-time strong solution to the diffusion system on an evolving surface with a boundary. One can see some textbooks on PDEs for mathematical analysis of the heat equations in the whole and half spaces. Note that the structure of the heat system is different for one-phase and multi-phase problems. Indeed, the solution $\theta_S$ is not necessarily equal to zero even if the initial value $\theta_0^S$ is zero

Finally, we introduce the results related to this paper. Favini-Goldstein-Goldstein-Romanelli \cite{FGGR02} and Vogt-Voigt \cite{VV03} studied the Wentzell boundary conditions on a boundary of a $C^1$-domain in $\mathbb{R}^d$. In \cite{FGGR02}, they proved that the Laplace operator with the Wentzell boundary condition generates an analytic semigroup on their function space. In \cite{VV03}, they characterized the Wentzell boundary conditions by using the measure theory. In this paper, we consider an evolution equation on the boundary.

The outline of this paper is as follows: In Section \ref{sect2}, we recall fundamental properties of the Laplace operators for $\mathbb{R}^3_+$, $\mathbb{R}^3_-$, and $\mathbb{R}^2$. In Section \ref{sect3}, we derive useful properties of the linear operator $L$ to construct strong solutions of our system. In Section \ref{sect4}, we apply maximal $L^2$-regularity, heat semigroups, and heat kernels to show the existence of a global-in-time strong solution to system \eqref{eq14}. In Section \ref{sect5}, we make use of a strong solution to \eqref{eq14} to construct a global-in-time strong solution to system \eqref{eq11}, and deduce the energy equality for our system. In Appendix (I), we derive our heat equations in $\mathbb{R}^3_+$, $\mathbb{R}^3_-$, and the interface $\mathbb{R}^2$ from an energetic point of view. In Appendix (II), we study useful properties of analytic semigroups.


\section{Preliminaries}\label{sect2}

Let us recall basic properties of the Laplace operators. Let $\kappa_A, \kappa_B, \tilde{\kappa}_S >0$. Define the linear operator $L_A$ in $L^2(\mathbb{R}^3_+)$, $L_B$ in $L^2(\mathbb{R}^3_-)$, and $L_S$ in $L^2(\mathbb{R}^2)$ as follows:
\begin{align*}
&\begin{cases}
L_A f_A = - \kappa_A \Delta f_A,\\
D ( L_A ) = W_0^{1,2} (\mathbb{R}^3_{+}) \cap W^{2,2} ( \mathbb{R}^3_{+}),
\end{cases}{ \ }
\begin{cases}
L_B f_B = - \kappa_B \Delta f_B,\\
D ( L_B ) = W_0^{1,2} (\mathbb{R}^3_{-}) \cap W^{2,2} ( \mathbb{R}^3_{-}),
\end{cases}\\
&\begin{cases}
L_S f_S = - \tilde{\kappa}_S \Delta_h f_S,\\
D ( L_S ) = W^{2,2} ( \mathbb{R}^2).
\end{cases}
\end{align*}
We write the semigroup whose generator is $-L_\natural$ as ${\rm{e}}^{- t L_\natural}$ $(\natural = A,B,S)$. We call $L_A$,$L_B$, $L_S$ the \emph{Laplace operators}, and ${\rm{e}}^{- t L_A}$, ${\rm{e}}^{- t L_B}$, ${\rm{e}}^{- t L_S}$ \emph{heat semigroups}. It is well-known that these Laplace operators have the following properties.
\begin{lemma}[Fundamental properties of the Laplace operators]\label{lem21}{ \ }\\
$(\rm{i})$ {\rm{[Self-adjoint operator]}} Each operator $L_A$, $L_B$, and $L_S$ is a self-adjoint operator in $L^2 (\mathbb{R}^3_{+})$, $L^2 (\mathbb{R}^3_{-})$, and, $L^2 (\mathbb{R}^2)$, respectively.\\
$(\rm{ii})$ {\rm{[Fractional power]}} $D (L_A^{1/2}) = W_0^{1,2} (\mathbb{R}^3_{+})$, $D (L_B^{1/2}) = W_0^{1,2} (\mathbb{R}^3_{-})$, and\\ $D (L_S^{1/2}) = W^{1,2} (\mathbb{R}^2)$. Moreover, for all $f_A \in D (L_A^{1/2})$, $f_B \in D (L_B^{1/2})$, and $f_S \in D (L_S^{1/2})$,
\begin{align*}
\Vert L_A^{1/2} f_A \Vert_{L^2 (\mathbb{R}^3_{+})} & = \sqrt{ \kappa_A } \Vert \nabla f_A \Vert_{L^2(\mathbb{R}^3_{+})},\\
\Vert L_B^{1/2} f_B \Vert_{L^2 (\mathbb{R}^3_{-})} & = \sqrt{ \kappa_B } \Vert \nabla f_B \Vert_{L^2(\mathbb{R}^3_{-})},\\
\Vert L_S^{1/2} f_S \Vert_{L^2 (\mathbb{R}^2)} & = \sqrt{ \tilde{\kappa}_S} \Vert \nabla_h f_S \Vert_{L^2(\mathbb{R}^2)}.
\end{align*}
$(\rm{iii})$ {\rm{[Elliptic regularity]}} There is $C >0$ such that for all $f_A \in D (L_A)$, $f_B \in D (L_B)$, and $f_S \in D (L_S)$,
\begin{align*}
\Vert f_A \Vert_{W^{2,2} ( \mathbb{R}^3_{+})} + \Vert \gamma_{+}[\partial_3 f_A ] \Vert_{L^2 ( \mathbb{R}^2)} & \leq C \left( \frac{1}{\kappa_A}\Vert L_A f_A \Vert_{L^2 (\mathbb{R}^3_{+})} + \Vert f_A \Vert_{L^2 (\mathbb{R}^3_{+})} \right),\\
\Vert f_B \Vert_{W^{2,2} ( \mathbb{R}^3_{-})} + \Vert \gamma_{-}[\partial_3 f_B ] \Vert_{L^2 ( \mathbb{R}^2)} & \leq C \left( \frac{1}{\kappa_B} \Vert L_B f_B \Vert_{L^2 (\mathbb{R}^3_{-})} + \Vert f_B \Vert_{L^2 (\mathbb{R}^3_{-})} \right),\\
\Vert \nabla_h f_S \Vert_{L^2 ( \mathbb{R}^2)} + \Vert \nabla_h^2 f_S \Vert_{L^2 (\mathbb{R}^2)} & \leq C \left( \frac{1}{\tilde{\kappa}_S} \Vert L_S f_S \Vert_{L^2 (\mathbb{R}^2)} + \Vert f_S \Vert_{L^2 (\mathbb{R}^2)} \right).
\end{align*}
Here $C$ is independent of $(\kappa_A , \kappa_B , \tilde{\kappa}_S)$, and $\gamma_{+}$, $\gamma_{-}$ are two trace operators such that $\gamma_{+}[\cdot] : W^{1,2} (\mathbb{R}^3_{+}) \to L^2 (\mathbb{R}^2)$ and $\gamma_{-}[\cdot] : W^{1,2} (\mathbb{R}^3_{-}) \to L^2 (\mathbb{R}^2)$.\\
$(\rm{iv})$ {\rm{[$C_0$-semigroup]}} Each operator $- L_A$, $-L_B$, and, $-L_S$ generates a contraction $C_0$-semigroup on $L^2 (\mathbb{R}^3_{+})$, $L^2 (\mathbb{R}^3_{-})$, and, $L^2 (\mathbb{R}^2)$, respectively.\\
$(\rm{v})$ {\rm{[Analytic semigroup]}} Each operator $- L_A$, $-L_B$, and $-L_S$ generates a bounded analytic semigroup on $L^2 (\mathbb{R}^3_{+})$, $L^2 (\mathbb{R}^3_{-})$, and, $L^2 (\mathbb{R}^2)$, respectively.\\
$(\rm{vi})$ {\rm{[Heat kernel]}} For each $f_A \in L^2(\mathbb{R}^3_+)$, $f_B \in L^2(\mathbb{R}^3_-)$, and $f_S \in L^2(\mathbb{R}^2)$, ${\rm{e}}^{- t L_A } f_A = G_A * f_A$, ${\rm{e}}^{- t L_B} f_B = G_B*f_B$, and ${\rm{e}}^{ - t L_S} f_S = G_S * f_S$, where $G_A = G_A (x,t)$ denotes the heat kernel for $-L_A$, $G_B = G_B (x,t)$ the heat kernel for $-L_B$, and $G_S = G_S (x_h ,t) = {\rm{exp}}(- \vert x_h \vert^2/{4t})/{ 4 \pi t}$.\\
$(\rm{vii})$ {\rm{[Maximal $L^p$-regularity]}} Let $1<p< \infty$. Let $\mathcal{F}_A \in L^p (0, \infty ; L^2 (\mathbb{R}^3_{+}))$, $\mathcal{F}_B \in L^p (0, \infty ; L^2 (\mathbb{R}^3_{-}))$, and $\mathcal{F}_S \in L^p (0, \infty ; L^2 (\mathbb{R}^2))$. Then there are unique functions $(U_A , U_B, U_S)$ satisfying
\begin{equation*}
\begin{cases}
\frac{d}{dt} U_A+ L_A U_A = \mathcal{F}_A \text{ on } (0, \infty),\\
\frac{d}{dt} U_B+ L_B U_B = \mathcal{F}_B \text{ on } (0, \infty),\\
\frac{d}{dt} U_S+ L_S U_S = \mathcal{F}_S \text{ on } (0, \infty),\\
U_A \vert_{t=0} = 0,{ \ }U_B \vert_{t=0}=0,{ \ }U_S\vert_{t=0} =0,
\end{cases}
\end{equation*}
and
\begin{align*}
\| d U_A/{dt} \|_{L^p(0, \infty ; L^2 (\mathbb{R}^3_{+}))} + \| L_A U_A  \|_{L^p (0, \infty ; L^2(\mathbb{R}^3_{+}))} & \leq K_A \| \mathcal{F}_A \|_{L^p (0, \infty ; L^2(\mathbb{R}^3_{+}))},\\
\| d U_B/{dt} \|_{L^p(0, \infty ; L^2 (\mathbb{R}^3_{-}))} + \| L_B U_B  \|_{L^p (0,\infty; L^2(\mathbb{R}^3_{-}))} & \leq K_B \| \mathcal{F}_B \|_{L^p (0,\infty; L^2(\mathbb{R}^3_{-}))},\\
\| d U_S/{dt} \|_{L^p(0,\infty ; L^2 (\mathbb{R}^2))} + \| L_S U_S  \|_{L^p (0,\infty ; L^2(\mathbb{R}^2))} & \leq  K_S \| \mathcal{F}_S \|_{L^p (0,\infty; L^2(\mathbb{R}^2))}.
\end{align*}
Here $K_A = K_A (\kappa_A,p)$, $K_B = K_B (\kappa_B,p)$, $K_S = K_S (\tilde{\kappa}_S,p) >0$ are independent of $(\mathcal{F}_A,\mathcal{F}_B, \mathcal{F}_S)$.

\end{lemma}
\noindent See \cite[Chapter II.3]{Soh01} for the assertions $(\rm{i})$, $(\rm{ii})$, \cite[Chapters 8,9]{GT98} for $(\rm{iii})$, \cite{Paz83}, \cite{DHP01} for $(\rm{iv})$, $(\rm{v})$, \cite{Uka87} for $(\rm{vi})$ and heat kernels for half spaces, and \cite{DHP01} for $(\rm{vii})$. See also \cite{DHP03} and \cite{KW04} for maximal $L^p$-regularity. From \cite{Des64}, we see that if a linear operator $- \mathcal{L}$ on a Hilbert space $\mathcal{H}$ generates a bounded analytic semigroup on $\mathcal{H}$ then $\mathcal{L}$ has maximal $L^p$-regularity.
\begin{proof}[Proof of Lemma \ref{lem21}]
We only prove that there is $C>0$ independent of $\kappa_A$ such that for all $f_A\in D (L_A)$
\begin{equation}\label{eq21}
\Vert \gamma_+[\partial_3 f_A ] \Vert_{L^2( \mathbb{R}^2)} \leq \frac{C}{\kappa_A} \Vert L_A f_A \Vert_{L^2 (\mathbb{R}^3_+)} + C \Vert f_A  \Vert_{L^2( \mathbb{R}^3_+)}.
\end{equation}
Set $L_{A_0} f = (1- \Delta) f$ and $D (L_{A_0}) = W^{1,2}_0 ( \mathbb{R}^3_+) \cap W^{2,2}(\mathbb{R}^3_+)$. Let $f_A \in D (L_A) = D (L_{A_0})$. Using the elliptic regularity(\cite[Chapter 9]{GT98}), we check that
\begin{align*}
\Vert \gamma_+[\partial_3 f_A] \Vert_{L^2 (\mathbb{R}^2)} & \leq C \Vert \partial_3 f_A \Vert_{W^{1,2} (\mathbb{R}^3_+)}\\
& \leq C \Vert f_A \Vert_{W^{2,2} (\mathbb{R}^3_+)}\\
& \leq C \Vert (1- \Delta ) f_A \Vert_{L^2 (\mathbb{R}^3_+)}\\
& \leq \frac{C}{\kappa_A} \Vert L_A f_A \Vert_{L^2 (\mathbb{R}^3_+)} + C \Vert f_A \Vert_{L^2(\mathbb{R}^3_+)}.
\end{align*}
Thus, we have \eqref{eq21}.
  \end{proof}

Next, we study the trace operators $\gamma_{\pm}$. Let us check that the trace operators $\gamma_+,\gamma_-$ can be considered as closed operators under some special conditions.
\begin{lemma}[Trace operators]\label{lem22}
Let $\{ f_A^m \}_{m \in \mathbb{N}} \subset D (L_A)$, $f_A \in D(L_A)$, $\{ f_B^m \}_{m \in \mathbb{N}} \subset D (L_B)$, $f_B \in D(L_B)$, and $f_S \in L^2(\mathbb{R}^2)$. Assume that
\begin{align*}
&\lim_{m \to \infty} \Vert f_A^m- f_A \Vert_{W^{2,2} (\mathbb{R}^3_+)} =0,\\
&\lim_{m \to \infty} \Vert f_B^m - f_B \Vert_{W^{2,2} (\mathbb{R}^3_-)} =0,\\
&\lim_{m \to \infty} \Vert (\kappa_A \gamma_+ [\partial_3 f_A^m] - \kappa_B \gamma_-[\partial_3 f_B^m]) - f_S \Vert_{L^2(\mathbb{R}^2)} = 0.
\end{align*}
Then
\begin{equation}\label{eq22}
f_S = \kappa_A \gamma_+ [\partial_3 f_A] - \kappa_B \gamma_-[\partial_3 f_B] \text{ in }L^2(\mathbb{R}^2).
\end{equation}
\end{lemma}

\begin{proof}[Proof of Lemma \ref{lem22}]
By assumptions and basic properties of the trace operators $\gamma_\pm$, we easily check that
\begin{multline*}
\Vert (\kappa_A \gamma_+ [\partial_3 f_A] - \kappa_B \gamma_-[\partial_3 f_B]) - f_S \Vert_{L^2(\mathbb{R}^2)}\\
 \leq  \Vert (\kappa_A \gamma_+ [\partial_3 f_A] - \kappa_B \gamma_- [\partial_3 f_B] )-  (\kappa_A \gamma_+ [\partial_3 f_A^m] - \kappa_B \gamma_- [\partial_3 f_B^m] ) \Vert_{L^2(\mathbb{R}^2)}\\ +  \Vert ( \kappa_A \gamma_+ [\partial_3 f_A^m] - \kappa_B \gamma_-[\partial_3 f_B^m]) - f_S \Vert_{L^2(\mathbb{R}^2)}\\
\leq C(\kappa_A) ( \Vert f_A^m - f_A \Vert_{W^{2,2} (\mathbb{R}^3_+)} +  C (\kappa_B) \Vert f_B^m - f_B \Vert_{W^{2,2} (\mathbb{R}^3_-)} \\
+ \Vert ( \kappa_A \gamma_+ [\partial_3 f_A^m] - \kappa_B \gamma_-[\partial_3 f_B^m]) - f_S \Vert_{L^2(\mathbb{R}^2)} \to 0 \text{ as } m \to \infty.
\end{multline*}
Therefore, we see \eqref{eq22}.
  \end{proof}

Finally, we introduce basic properties of analytic semigroups and one useful lemma.
\begin{lemma}\label{lem23}
Let $\mathcal{H}$ be a Hilbert space and $\Vert \cdot \Vert_{\mathcal{H}}$ its norm. Let $\mathcal{L}: D (\mathcal{L}) (\subset \mathcal{H}) \to \mathcal{H}$ be a linear operator on $\mathcal{H}$. Assume that $- \mathcal{L}$ generates a bounded analytic semigroup on $\mathcal{H}$. Let $V_0 \in H$ and $\mathcal{F} \in C^{\eta}_{loc} ((0, T) ; \mathcal{H}) \cap L^2(0,T;\mathcal{H})$ for some $0 < \eta \leq 1/2$ and $T \in (0 , \infty ) $. Then system
\begin{equation*}
\begin{cases}
\frac{d}{dt}V + \mathcal{L} V = \mathcal{F} \text{ on }(0, T),\\
V \vert_{t=0} = V_0,
\end{cases}
\end{equation*}
admits a unique strong solution $V$ in
\begin{equation*}
C([0, T ) ; \mathcal{H}) \cap C ((0, T); D (\mathcal{L}) ) \cap C^1 ((0,T); \mathcal{H}),
\end{equation*}
and $V$ is written by 
\begin{equation*}
V (t) = {\rm{e}}^{- t \mathcal{L}} V_0 + \int_0^t {\rm{e}}^{- (t- \tau ) \mathcal{L}} \mathcal{F} (\tau ) { \ }d \tau \text{ for } 0 < t <T.
\end{equation*}
Moreover, assume in addition that $V_0 \in D (\mathcal{L}^{1/2})$. Then
\begin{equation}\label{eq23}
V,  dV/{dt}, \mathcal{L}V \in C^{\eta/2}_{loc} ((0, T) ; \mathcal{H}).
\end{equation}
\end{lemma}
\noindent By using an argument in \cite[Chapter 4]{Paz83} or \cite{K22}, we can prove Lemma \ref{lem23}. For the readers, we derive \eqref{eq23} in Appendix (II). 
\begin{lemma}\label{lem24}
Let $\beta >0$ and $f_S \in L^2 (\mathbb{R}^2) $. Then
\begin{align*}
\Vert {\rm{e}}^{- \beta x_3} f_S (x_h) \Vert_{L^2 (\mathbb{R}^3_{+})} & \leq \frac{1}{\sqrt{2 \beta}} \Vert f_S \Vert_{L^2 (\mathbb{R}^2)},\\
\Vert {\rm{e}}^{ \beta x_3} f_S (x_h) \Vert_{L^2 (\mathbb{R}^3_{-})} & \leq \frac{1}{\sqrt{2 \beta}} \Vert f_S \Vert_{L^2 (\mathbb{R}^2)}.
\end{align*}
\end{lemma}
\begin{proof}[Proof of Lemma \ref{lem24}]
Fix $\beta >0$ and $f_S \in L^2 (\mathbb{R}^2) $. Since
\begin{equation*}
\left( \int_0^\infty {\rm{e}}^{ - 2 \beta x_3} d x_3 \right)^{1/2} = \frac{1}{ \sqrt{2\beta} } \text{ and } \left( \int_{-\infty}^0 {\rm{e}}^{ 2 \beta x_3} d x_3 \right)^{1/2} = \frac{1}{ \sqrt{2\beta} },
\end{equation*}
we easily check that
\begin{align*}
\Vert {\rm{e}}^{- \beta x_3} f_S (x_h) \Vert_{L^2 (\mathbb{R}^3_{+})} & \leq \Vert {\rm{e}}^{ - \beta \cdot } \Vert_{L^2(0,\infty)} \Vert f_h \Vert_{L^2(\mathbb{R}^2)} = \frac{1}{\sqrt{2 \beta}} \Vert f_S \Vert_{L^2 (\mathbb{R}^2)},\\
\Vert {\rm{e}}^{ \beta x_3} f_S (x_h) \Vert_{L^2 (\mathbb{R}^3_{-})} & \leq \Vert {\rm{e}}^{ \beta \cdot } \Vert_{L^2(-\infty, 0 )} \Vert f_h \Vert_{L^2(\mathbb{R}^2)} = \frac{1}{\sqrt{2 \beta}} \Vert f_S \Vert_{L^2 (\mathbb{R}^2)}.
\end{align*}
Therefore, the lemma follows.
\end{proof}


\section{Function Spaces in $\mathbb{R}^3_+$, $\mathbb{R}^3_-$ with Interface $\mathbb{R}^2$}\label{sect3}

Let us introduce our function spaces and study a key linear operator on the function space. Let $\kappa_A , \kappa_B, \tilde{\kappa}_S, \beta >0$. Set
\begin{equation*}
H = \{ f = { }^t (f_A , f_B , f_S ) \in L^2 ( \mathbb{R}^3_{+}) \times L^2 (\mathbb{R}^3_{-}) \times L^2 (\mathbb{R}^2) ; { \ } \Vert f \Vert_{H} < + \infty \}
\end{equation*}
with
\begin{equation*}
\Vert f \Vert_{H} = \Vert { }^t ( f_A , f_B , f_S ) \Vert_H = \bigg( \Vert f_A \Vert_{L^2 ( \mathbb{R}^3_{+})}^2 + \Vert f_B \Vert_{L^2 ( \mathbb{R}^3_{-})}^2 + \Vert f_S \Vert_{L^2 ( \mathbb{R}^2)}^2 \bigg)^{1/2}.
\end{equation*}
Moreover, we define for all $f = { }^t (f_A , f_B, f_S)$, $g = { }^t( g_A , g_B , g_S) \in H$,
\begin{align*}
\left< f , g \right>_H  & = \left< { }^t( f_A , f_B , f_S) , { }^t( g_A , g_B, g_S)  \right>_H\\
&:= \left< f_A , g_A \right>_{L^2(\mathbb{R}^3_{+})} +  \left< f_B , g_B \right>_{L^2(\mathbb{R}^3_{-})} + \left< f_S , g_S \right>_{L^2(\mathbb{R}^2)}\\
& = \int_{\mathbb{R}^3_{+}} f_A \overline{g_A} { \ }d x + \int_{\mathbb{R}^3_{-}} f_B \overline{g_B} { \ }d x + \int_{\mathbb{R}^2} f_S \overline{g_S} { \ }d x_h,
\end{align*}
where $\overline{f_\natural}$ is the complex conjugate function of $f_\natural$.
We easily see that $H$ is a Hilbert space since $L^2(\mathbb{R}^3_{+})$, $L^2(\mathbb{R}^3_{-})$, and $L^2(\mathbb{R}^2)$ are Hilbert spaces. 

Now we define a key linear operator in $H$. Set
\begin{equation*}
\begin{cases}
L f = { }^t( L_A f_A, L_B f_B , L_S f_S ),\\
D (L) = D (L_A ) \times D (L_B ) \times D (L_S),
\end{cases}
\end{equation*}
and
\begin{equation*}
\begin{cases}
L^{1/2} f = { }^t( L^{1/2}_A f_A, L^{1/2}_B f_B , L^{1/2}_S f_S ),\\
D (L^{1/2}) = D (L^{1/2}_A ) \times D (L^{1/2}_B ) \times D (L^{1/2}_S).
\end{cases}
\end{equation*}
Here $L_A$, $L_B$, $L_S$ are the operators defined by Section \ref{sect2}. From Lemma \ref{lem21}, we have the lemma.
\begin{lemma}[Fundamental properties of $L$]\label{lem31}{ \ }\\
$(\rm{i})$ The operator $- L $ generates a contraction $C_0$-semigroup on $H$.\\
$(\rm{ii})$ The operator $- L$ generates a bounded analytic semigroup on $H$.\\
$(\rm{iii})$ For $V_0 = { }^t (V_0^A , V_0^B , V_0^S) \in H$,
\begin{align*}
{\rm{e}}^{- t L } V_0 &= { }^t ({\rm{e}}^{-tL_A} V_0^A , {\rm{e}}^{-tL_B} V_0^B , {\rm{e}}^{- tL_S} V_0^S )\\
 &= { }^t (G_A * V_0^A, G_B*V_0^B , G_S*V_0^S) := G*V_0,
\end{align*}
where ${\rm{e}}^{- t L_\natural}$ is a heat semigroup and $G_\natural$ is a heat kernel $(\natural = A,B,S)$.\\
$(\rm{iv})$ The operator $L$ has maximal $L^p$-regularity.
\end{lemma}

We now study the key properties of the operator $L$ to construct a strong solution of our system.
\begin{lemma}\label{lem32}
Let $V_0 = { }^t (V_0^A , V_0^B , V_0^S) \in H$ and $\mathcal{F} = { }^t (\mathcal{F}_A , \mathcal{F}_B , \mathcal{F}_S) $\\ $ \in C_{loc}^\eta ((0, T); H) \cap L^2(0,T ; H)$ for some $0 < \eta \leq 1/2$ and $T \in (0, \infty )$. Then system
\begin{equation}\label{eq31}
\begin{cases}
\frac{d}{dt}V + L V = \mathcal{F} \text{ on } (0, T),\\
V \vert_{t=0} = V_0,
\end{cases}
\end{equation}
admits a unique strong solution $V$ in
\begin{equation}\label{eq32}
C([0, T) ; H) \cap C ((0, T); D (L) ) \cap C^1 ((0,T); H),
\end{equation}
and $V$ is written by 
\begin{equation}\label{eq33}
V (t) = {\rm{e}}^{- t L} V_0 + \int_0^t {\rm{e}}^{- (t - \tau ) L} \mathcal{F} (\tau ) { \ }d \tau \text{ for } 0 < t < T.
\end{equation}
Moreover, the following three assertions hold:\\
\noindent $(\rm{i})$
\begin{equation}\label{eq34}
\sup_{0 \leq t <T}\Vert V (t) \Vert_H \leq \Vert V_0 \Vert_H +T^{1/2} \Vert \mathcal{F} \Vert_{L^2(0,T;H)}.
\end{equation}
\noindent $(\rm{ii})$ Assume in addition that $V_0 \in D ( L^{1/2})$. Then
\begin{equation*}
V,  dV/{dt}, L V \in C^{\eta/2}_{loc} ((0, T); H).
\end{equation*}
\noindent $(\rm{iii})$ Assume in addition that $V_0 \in D ( L^{1/2})$. Then
\begin{multline}\label{eq35}
\Vert d V/ {dt} \Vert_{L^2(0, T ;H)} + \Vert L V \Vert_{L^2(0, T ;H)} \leq 2 \Vert L^{1/2} V_0 \Vert_{H} + K_A \| \mathcal{F}_A \|_{L^2 (0, T ; L^2(\mathbb{R}^3_{+}))}\\
 + K_B \| \mathcal{F}_B \|_{L^2 (0,T; L^2(\mathbb{R}^3_{-}))} + K_S \| \mathcal{F}_S \|_{L^2 (0,T; L^2(\mathbb{R}^2))}.
\end{multline}
Here $K_A$, $K_B$, $K_S$ are the three positive constants appearing in $(\rm{vii})$ of Lemma \ref{lem21}.
\end{lemma}

\begin{proof}[Proof of Lemma \ref{lem32}]
Fix $V_0 = { }^t (V_0^A , V_0^B , V_0^S) \in H$ and $\mathcal{F} = { }^t (\mathcal{F}_A , \mathcal{F}_B , \mathcal{F}_S) \in L^2(0,T;H) \cap C_{loc}^\eta ((0, T); H)$ for some $0 < \eta \leq 1/2$ and $T \in (0, \infty )$. Since $-L$ generates an analytic semigroup on $H$, we see that \eqref{eq33} is a unique strong solution of system \eqref{eq31} and in \eqref{eq32}. Since ${\rm{e}}^{- t L}$ is a contraction $C_0$-semigroup on $H$, we apply the Cauchy-Schwarz inequality to derive \eqref{eq34}. From Lemma \ref{lem23}, we see that $(\rm{ii})$ holds. 

We now show $(\rm{iii})$. Assume that $V_0 \in D (L^{1/2})$. Set $W = W (t) =$\\ $ { }^t (W_A,W_B,W_S) = {\rm{e}}^{- t L} V_0$, and $U = U(t) = { }^t (U_A,U_B,U_S) = V (t) - W (t)$. Then we see that

\begin{equation}\label{eq36}
\begin{cases}
\frac{d}{dt}W + L W = 0 \text{ on } (0, T),\\
W \vert_{t=0} = V_0,
\end{cases}
\end{equation}
\begin{equation}\label{eq37}
\begin{cases}
\frac{d}{dt}U + L U = \mathcal{F} \text{ on } (0, T),\\
U \vert_{t=0} = 0.
\end{cases}
\end{equation}
Applying the maximal $L^2$-regularity($(\rm{vii})$ of Lemma \ref{lem21}) into \eqref{eq37}, we see that
\begin{multline}\label{eq38}
\Vert d U/ {dt} \Vert_{L^2(0, T ;H)} + \Vert L U \Vert_{L^2(0, T ;H)} \leq K_A \| \mathcal{F}_A \|_{L^2 (0, T ; L^2(\mathbb{R}^3_{+}))}\\
 + K_B \| \mathcal{F}_B \|_{L^2 (0,T; L^2(\mathbb{R}^3_{-}))} + K_S \| \mathcal{F}_S \|_{L^2 (0,T; L^2(\mathbb{R}^2))}.
\end{multline}

Next, we prove that for $0 < t < \infty$
\begin{equation}\label{eq39}
L {\rm{e}}^{- t L} V_0 = L^{1/2} {\rm{e}}^{- t L}L^{1/2} V_0 \text{ in }H.
\end{equation}
By definition, we check that
\begin{align*}
L {\rm{e}}^{- t L}V_0 & = { }^t (L_A {\rm{e}}^{- t L_A} V_0^A , L_B {\rm{e}}^{- t L_B} V_0^B, L_S {\rm{e}}^{- t L_S }V_0^S )\\
& = { }^t (L_A^{1/2} {\rm{e}}^{- t L_A} L_A^{1/2} V_0^A , L_B^{1/2} {\rm{e}}^{- t L_B} L_B^{1/2} V_0^B , L_S^{1/2} {\rm{e}}^{- t L_S} L_S^{1/2} V_0^S)\\
 &= L^{1/2} {\rm{e}}^{- t L}L^{1/2} V_0.
\end{align*}
Thus, we have \eqref{eq39}.

From Lemma \ref{lem21}, we observe that
\begin{multline*}
\frac{1}{2} \frac{d}{dt} \Vert W (t) \Vert_H^2 = \left< d W /{dt} , W \right>_H = \left< - L W , W \right>_H\\
 = \left< -L_A W_A, W_A \right>_{L^2 (\mathbb{R}^3_+)} + \left< - L_B W_B , W_B \right>_{L^2(\mathbb{R}^3_-)} + \left< - L_S W_S,W_S \right>_{L^2(\mathbb{R}^2)}\\
 = - \Vert L_A^{1/2} W_A \Vert_{L^2(\mathbb{R}^3_+)}^2 - \Vert L_B^{1/2} W_B \Vert_{L^2(\mathbb{R}^3_-)}^2 - \Vert L_S^{1/2} W_S \Vert_{L^2(\mathbb{R}^2)}^2 = - \Vert L^{1/2} W \Vert_H^2.
\end{multline*}
Integrating with respect to time $t$, we have
\begin{equation*}
\Vert W (t) \Vert_H^2 + \int_0^t \Vert L^{1/2} W (\tau) \Vert_H^2 { \ }d \tau = \Vert W(0) \Vert_H^2.
\end{equation*}
This shows that
\begin{equation*}
\Vert {\rm{e}}^{- t L} V_0 \Vert_H^2 + \int_0^t \Vert L^{1/2} {\rm{e}}^{- \tau L} V_0 \Vert_H^2 { \ }d \tau = \Vert V_0 \Vert_H^2.
\end{equation*}
Since $L^{1/2} V_0 \in H$, we have
\begin{equation}\label{Eq3010}
\Vert {\rm{e}}^{- t L} L^{1/2}V_0 \Vert_H^2 + \int_0^t \Vert L^{1/2} {\rm{e}}^{- \tau L} L^{1/2} V_0 \Vert_H^2 { \ }d \tau = \Vert L^{1/2} V_0 \Vert_H^2.
\end{equation}
By \eqref{eq39} and \eqref{Eq3010}, we see that
\begin{equation*}
\int_0^t \Vert L {\rm{e}}^{- \tau L} V_0 \Vert_H^2 { \ }d \tau = \int_0^t \Vert L^{1/2} {\rm{e}}^{- \tau L} L^{1/2} V_0 \Vert_H^2 { \ }d \tau \leq \Vert L^{1/2} V_0 \Vert_H^2.
\end{equation*}
Thus, we have
\begin{equation}\label{Eq3011}
\Vert d W/ {dt} \Vert_{L^2(0, T ;H)} + \Vert L W \Vert_{L^2(0, T ;H)} \leq 2 \Vert L^{1/2} V_0 \Vert_H.
\end{equation}
Combining \eqref{eq38} and \eqref{Eq3011}, we have \eqref{eq35}. Therefore, the lemma follows.
  \end{proof}


\section{Existence of Strong Solutions}\label{sect4}
In this section, we show the existence of strong solutions to system \eqref{eq14}. We apply the maximal $L^2$-regularity for system \eqref{eq31}, and nice properties of both heat semigroups and kernels to construct local-in-time and global-in-time strong solutions to our system.

Let $\kappa_A, \kappa_B , \tilde{\kappa}_S, \alpha_S , \beta >0 $, $T \in (0,\infty)$, and let $L_A$, $L_B$, $L_S$, and $L$ be the three operators defined by Section \ref{sect2}, and the operator defined by Section \ref{sect3}.

Let us introduce our function spaces. Set
\begin{align*}
&L^2((0,T) \times \mathbb{R}^3_{+,-,0}) := L^2((0,T) \times \mathbb{R}^3_+) \times L^2 ((0,T) \times \mathbb{R}^3_-) \times L^2((0,T) \times \mathbb{R}^2),\\
&X_T := \{ \varphi = { }^t(\varphi_A , \varphi_B , \varphi_S ) \in C( [0,T); H ) ; { \ } \Vert \varphi \Vert_{X_T} < + \infty \}
\end{align*}
with
\begin{equation*}
\Vert \varphi \Vert_{X_T} := \sup_{0 \leq t <T} \Vert \varphi \Vert_{H} + \Vert d \varphi/{dt} \Vert_{L^2 (0,T; H)} + \Vert L \varphi \Vert_{L^2(0,T; H)}.
\end{equation*}
Moreover,
\begin{align*}
&L^2_{loc}(0,\infty ; D (L)) := \{ \varphi; { \ }\varphi \in L^2(0,T;D(L)) \text{ for each fixed }T>0 \},\\
&W^{1,2}_{loc}(0,\infty ; H) := \{ \varphi; { \ }\varphi \in L^2(0,T; H) \text{ for each fixed }T>0 \},\\
&L^2_{loc} ( \mathbb{R}_+ \times \mathbb{R}^3_{+,-,0}) := \{ \varphi ;{ \ } \varphi \in L^2((0,T) \times \mathbb{R}^3_{+,-,0}) \text{ for each fixed }T >0  \}.
\end{align*}

Applying the linear operator $L$ into system \eqref{eq14}, we have
\begin{equation}\label{eq41}
\begin{cases}
\frac{d}{d t} v + L v = F (v) { \ }\text{ on }(0,T),\\
v \vert_{t=0} = v_0,
\end{cases}
\end{equation}
where $v = { }^t (v_A , v_B , v_S) \in X_T$ and $v_0 = { }^t(v_0^A , v_0^B , v_0^S) \in H$. Here
\begin{equation*} F (v) = 
\begin{pmatrix}
F_1 (v)\\
F_2 (v)\\
F_3 (v)
\end{pmatrix} := \begin{pmatrix}
 - ( d v_S /{dt} ) {\rm{e}}^{- \beta x_3} + (\kappa_A \Delta_h v_S ) {\rm{e}}^{- \beta x_3} + \beta^2 \kappa_A v_S {\rm{e}}^{- \beta x_3}\\
 - ( d v_S / {dt} ) {\rm{e}}^{\beta x_3} + ( \kappa_B \Delta_h v_S ) {\rm{e}}^{\beta x_3} + \beta^2 \kappa_B v_S {\rm{e}}^{ \beta x_3}\\
 \frac{\kappa_A}{\alpha_S} \gamma_{+}[ \partial_3 v_A] - \frac{\kappa_B}{\alpha_S} \gamma_{-}[ \partial_3 v_B] - \frac{\beta}{\alpha_S} (\kappa_A + \kappa_B) v_S
\end{pmatrix}.
\end{equation*}

The aim of this section is to prove the following theorem.
\begin{theorem}[Existence of a global-in-time strong solution]\label{thm41}
There are two positive constants $\alpha_0 = \alpha_0 ( \tilde{\kappa}_S)$ and $\beta_0 = \beta_0 (\kappa_A , \kappa_B , \tilde{\kappa}_S )$ such that if $\alpha_S > \alpha_0$ and $\beta > \beta_0$ then for each $v_0 = { }^t (v_0^A , v_0^B , v_0^S) \in D (L^{1/2})$, system \eqref{eq41} admits a unique global-in-time strong solution $v$ in 
\begin{equation*}
C ([0,\infty) ; H ) \cap L^2_{loc} (0,\infty; D(L)) \cap W_{loc}^{1,2} (0,\infty ; H) \cap L^2_{loc} ( \mathbb{R}_+ \times \mathbb{R}^3_{+,-,0}),
\end{equation*}
satisfying
\begin{equation*}
\lim_{t \to 0 + 0} v (t) = v_0 \text{ in }H. 
\end{equation*}
\end{theorem}
\noindent In order to prove Theorem \ref{thm41}, we derive some properties of $F (v)$ when $v \in X_T$ in subsection \ref{subsec41}. In subsection \ref{subsec42}, we study the uniqueness of the strong solutions to system \eqref{eq41}. In subsection \ref{subsec43}, we construct a local-in-time strong solution to our system. In subsection \ref{subsec44}, we show the existence of a global-in-time strong solution to \eqref{eq41}.

\subsection{Properties of $F(\cdot)$}\label{subsec41}
Let us study some properties of $F (\cdot)$.
\begin{lemma}[Properties of $F(\cdot)$]\label{lem42}{ \ }\\
$(\rm{i})$ Assume that $\varphi \in X_T$. Then $F (\varphi) \in L^2(0,T;H)$. Moreover, there is $C_*>0$ independent of $(\kappa_A , \kappa_B , \tilde{\kappa}_S, \alpha_S, \beta ,T, \varphi )$ such that
\begin{align}
\Vert F_1 (\varphi) \Vert_{L^2(0,T; L^2(\mathbb{R}^3_{+}))} & \leq C_* \bigg( \frac{1}{ \sqrt{\beta}} + \frac{\kappa_A}{\tilde{\kappa}_S \sqrt{\beta}} + T^{1/2} \kappa_A \frac{\beta^2}{\sqrt{\beta}} \bigg) \Vert \varphi \Vert_{X_T},\label{eq42}\\
\Vert F_2 (\varphi) \Vert_{L^2(0,T; L^2(\mathbb{R}^3_{-}))} & \leq C_* \bigg( \frac{1}{ \sqrt{\beta}} + \frac{\kappa_B}{\tilde{\kappa}_S \sqrt{\beta}} + T^{1/2} \kappa_B \frac{\beta^2}{\sqrt{\beta}} \bigg) \Vert \varphi \Vert_{X_T} ,\label{eq43}\\
\Vert F_3 (\varphi) \Vert_{L^2(0,T; L^2(\mathbb{R}^2))} & \leq C_* \bigg( \frac{1}{\alpha_S} + T^{1/2} \frac{(\kappa_A + \kappa_B)(\beta + 1)}{\alpha_S} \bigg)  \Vert \varphi \Vert_{X_T}.\label{eq44}
\end{align}
\noindent $(\rm{ii})$ Let $0 < \eta \leq 1/2$. Assume that $\varphi \in X_T$ and that
\begin{equation*}
\varphi , L \varphi, d\varphi/{dt} \in C^{\eta}_{loc} ((0, T) ; H).
\end{equation*}
Then
\begin{equation*}
F (\varphi) \in C^{\eta}_{loc} ((0, T) ; H).
\end{equation*}
\end{lemma}

\begin{proof}[Proof of Lemma \ref{lem42}]
We first show $(\rm{i})$. By Lemma \ref{lem24}, we check that
\begin{multline*}
\Vert F_1 (\varphi) \Vert_{L^2 ( \mathbb{R}^3_{+})} \\
\leq \frac{1}{\sqrt{2 \beta}} \Vert  d \varphi_S /{dt} \Vert_{L^2(\mathbb{R}^2)} + \frac{\kappa_A}{ \tilde{\kappa}_S \sqrt{2 \beta}} \Vert L_S \varphi_S \Vert_{L^2 (\mathbb{R}^2)} + \frac{\beta^2 \kappa_A}{\sqrt{2 \beta}}   \Vert \varphi_S \Vert_{L^2 ( \mathbb{R}^2)}.
\end{multline*}
This gives
\begin{multline*}
\Vert F_1 (\varphi) \Vert_{L^2(0,T; L^2 ( \mathbb{R}^3_{+}) ) } \leq C\bigg( \frac{1}{\sqrt{2 \beta}} \Vert  d \varphi_S /{dt} \Vert_{L^2(0,T;L^2(\mathbb{R}^2))}\\ + \frac{ \kappa_A}{\tilde{\kappa}_S \sqrt{2 \beta}} \Vert L_S \varphi_S \Vert_{L^2(0,T; L^2 (\mathbb{R}^2))} + \frac{\beta^2 \kappa_A}{\sqrt{2 \beta}} \Vert \varphi_S \Vert_{L^2(0,T;L^2 ( \mathbb{R}^2))} \bigg)\\
\leq C \bigg( \frac{1}{ \sqrt{\beta}} + \frac{\kappa_A}{ \tilde{\kappa}_S \sqrt{\beta}} + T^{1/2} \kappa_A \frac{\beta^2}{\sqrt{\beta}} \bigg) \Vert \varphi \Vert_{X_T}.
\end{multline*}
Here we used the fact that
\begin{equation*}
\Vert \varphi_S \Vert_{L^2(0,T;L^2(\mathbb{R}^2))} \leq  \sup_{0 \leq t < T } \Vert \varphi_S \Vert_{L^2(\mathbb{R}^2)} T^{1/2} \leq T^{1/2} \Vert \varphi \Vert_{X_T}.
\end{equation*}
Therefore, we see \eqref{eq42}. Similarly, we have \eqref{eq43}. By assertion $(\rm{iii})$ of Lemma \ref{lem21}, we check that
\begin{multline*}
\Vert F_3 (\varphi) \Vert_{L^2 ( \mathbb{R}^2)} \leq \frac{C}{\alpha_S} \Vert L_A \varphi_A \Vert_{L^2(\mathbb{R}^3_{+})} + \frac{C\kappa_A}{\alpha_S} \Vert \varphi_A \Vert_{L^2(\mathbb{R}^3_{+})} + \frac{C}{\alpha_S} \Vert L_B \varphi_B \Vert_{L^2 (\mathbb{R}^3_{-})}\\
 + \frac{C \kappa_B}{\alpha_S}  \Vert \varphi_B \Vert_{L^2 (\mathbb{R}^3_{-})} + \frac{\beta}{\alpha_S} (\kappa_A + \kappa_B ) \Vert \varphi_S \Vert_{L^2 ( \mathbb{R}^2)}.
\end{multline*}
This gives \eqref{eq44}. From \eqref{eq42}-\eqref{eq44}, we find that $F (\varphi) \in L^2(0,T;H)$ when $\varphi \in X_T$. Therefore, we see $(\rm{i})$.

Next, we show $(\rm{ii})$. Fix $0 < \varepsilon  <T$. By assumption, there is $C >0$ such that for all $t_1,t_2(\varepsilon \leq t_1 \leq t_2 < T)$
\begin{multline}\label{eq45}
\Vert \varphi (t_2) - \varphi (t_1) \Vert_{H} + \Vert L \varphi (t_2) - L \varphi (t_1) \Vert_{H} + \Vert {d \varphi}/{dt} (t_2) - {d\varphi}/{dt} (t_1) \Vert_{H}\\
 \leq C (t_2 - t_1)^{\eta} .
\end{multline}
Fix $t_1,t_2$ such that $\varepsilon \leq t_1 \leq t_2 <T$. By definition, we have
\begin{multline*}
\Vert F (\varphi(t_2)) - F (\varphi(t_1)) \Vert_H^2 = \Vert F_1(\varphi(t_2)) - F_1 (\varphi(t_1)) \Vert_{L^2 (\mathbb{R}^3_{+})}^2\\ + \Vert F_2 (\varphi(t_2)) - F_2 (\varphi(t_1)) \Vert^2_{L^2 (\mathbb{R}^3_{-})} + \Vert F_3 (\varphi(t_2)) - F_3 (\varphi(t_1)) \Vert^2_{L^2 (\mathbb{R}^2)}.
\end{multline*}
Applying an argument similar to show \eqref{eq42} with \eqref{eq45}, we find that
\begin{align*}
\Vert F_1 ( \varphi(t_2) ) - F_1 (\varphi (t_1)) \Vert_{L^2 (\mathbb{R}^3_+)} \leq C \Vert {d \varphi_S}/{dt}(t_2) - {d \varphi_S}/{dt} (t_1) \Vert_{L^2 (\mathbb{R}^2)}\\ + C \Vert L_S \varphi_S (t_2) - L_S \varphi_S (t_1) \Vert_{L^2 (\mathbb{R}^2)} + C \Vert \varphi_S (t_2) - \varphi_S (t_1) \Vert_{L^2 (\mathbb{R}^2)}\\
\leq C (T,\kappa_A , \tilde{\kappa}_S, \beta ) (t_2 - t_1)^{\eta}. 
\end{align*}
Similarly, we see that
\begin{equation*}
\Vert F_2 ( \varphi(t_2) ) - F_2 (\varphi (t_1)) \Vert_{L^2 (\mathbb{R}^3_{-})} \leq C(T, \kappa_B , \tilde{\kappa}_S , \beta ) (t_2 - t_1)^{\eta}. 
\end{equation*}
By \eqref{eq45} and an argument similar to derive \eqref{eq44}, we observe that
\begin{equation*}
\Vert F_3 ( \varphi (t_2) ) - F_3 (\varphi (t_1)) \Vert_{L^2 (\mathbb{R}^2)} \leq C(T, \kappa_A , \kappa_B , \tilde{\kappa}_S, \alpha_S , \beta ) (t_2 - t_1)^{\eta}. 
\end{equation*}
As a result, we see that $F (\varphi) \in C^\eta ([\varepsilon , T);H)$ for each fixed $\varepsilon \in (0,T)$. Therefore, we see $(\rm{ii})$.
  \end{proof}

Applying Lemmas \ref{lem32} and \ref{lem42}, we have the following lemma.
\begin{lemma}\label{lem43}
Let $v_0 = { }^t (v_0^A , v_0^B , v_0^S) \in D (L^{1/2})$ and $\varphi = { }^t (\varphi_A , \varphi_B , \varphi_S ) \in X_T$. Assume that
\begin{equation*}
F (\varphi) \in C^{\eta}_{loc} ((0,T); H) \text{ for some }0 < \eta \leq 1/2.
\end{equation*}
Then system
\begin{equation*}
\begin{cases}
\frac{d}{dt} v + Lv = F (\varphi) \text{ on } (0,T),\\
v \vert_{t =0} = v_0,
\end{cases}
\end{equation*}
admits a unique strong solution $v$ in 
\begin{equation*}
C ([0,T);H) \cap C((0,T);D (L)) \cap C^1 ((0,T);H).
\end{equation*}
Moreover, $v$ satisfies the following properties:\\
$(\rm{i})$ For $0 < t <T$, $v$ is written by
\begin{equation*}
v (t) = {\rm{e}}^{- t L} v_0 + \int_0^t {\rm{e}}^{- ( t - \tau ) L} F (\varphi ( \tau )) { \ }d \tau.
\end{equation*}
$(\rm{ii})$ $v, dv/{dt}, L v \in C_{loc}^{\eta/2} ((0,T); H)$.\\
$(\rm{iii})$ If $T \leq 1$, then
\begin{multline*}
\Vert v \Vert_{X_T} \leq \Vert v_0 \Vert_{H} + 2 \Vert L^{1/2} v_0 \Vert_{H} + C (\kappa_A , \kappa_B , \tilde{\kappa}_S , \alpha_S , \beta) T^{1/2} \Vert \varphi \Vert_{X_T}\\
 + \bigg\{ \frac{K_S C_*}{\alpha_S} + K_A C_* \bigg( \frac{1}{\sqrt{\beta}} + \frac{\kappa_A}{\tilde{\kappa}_S \sqrt{\beta}} \bigg) + K_B C_* \bigg( \frac{1}{\sqrt{\beta}} + \frac{\kappa_B}{\tilde{\kappa}_S \sqrt{\beta}} \bigg) \bigg\} \Vert \varphi \Vert_{X_T}.
\end{multline*}
Here $K_A$, $K_B$, $K_S$ are the three positive constants appearing in $(\rm{vii})$ of Lemma \ref{lem21}, and $C_*$ is the positive constant appearing in $(\rm{i})$ of Lemma \ref{lem42}.
\end{lemma}

\begin{proof}[Proof of Lemma \ref{lem43}]
Since $F (\varphi) \in L^2(0,T;H)$ when $\varphi \in X_T$ from Lemma \ref{lem42}, we apply Lemma \ref{lem32} to deduce $(\rm{i})$ and $(\rm{ii})$. We now prove $(\rm{iii})$. By definition, we have
\begin{equation}\label{eq46}
\Vert v \Vert_{X_T} = \sup_{0 \leq t < T} \Vert v(t) \Vert_H + \Vert d v/{dt} \Vert_{L^2(0,T;H)} + \Vert L v \Vert_{L^2(0,T;H)}.
\end{equation}
Since ${\rm{e}}^{- t L}$ is a contraction $C_0$-semigroup on $H$, we apply property $(\rm{i})$, the Cauchy-Schwarz inequality, and Lemma \ref{lem42}, we check that
\begin{align*}
\Vert v(t) \Vert_H & \leq \Vert {\rm{e}}^{- t L} v_0 \Vert_H + \left\Vert \int_0^t {\rm{e}}^{- (t - \tau ) L } F (\varphi(\tau)) { \ }d\tau \right\Vert_H\\
& \leq \Vert v_0 \Vert_{H} + T^{1/2} \Vert F (\varphi ) \Vert_{L^2(0,T;H)}\\
& \leq \Vert v_0 \Vert_H + C (\kappa_A , \kappa_B, \tilde{\kappa}_S , \alpha_S , \beta) T^{1/2} \Vert \varphi \Vert_{X_T}.
\end{align*}
Note that $T \leq 1$. Thus, we have
\begin{equation}\label{eq47}
\sup_{0 \leq t <T} \Vert v(t) \Vert_H \leq \Vert v_0 \Vert_H + C (\kappa_A , \kappa_B, \tilde{\kappa}_S , \alpha_S , \beta) T^{1/2} \Vert \varphi \Vert_{X_T}.
\end{equation}
Using $(\rm{iii})$ of Lemma \ref{lem32} and Lemma \ref{lem42}, we observe that
\begin{multline}\label{eq48}
 \Vert d v/{dt} \Vert_{L^2(0,T;H)} + \Vert L v \Vert_{L^2(0,T;H)} \leq 2 \Vert L^{1/2} v_0 \Vert_H + K_A \Vert F_1 (\varphi) \Vert_{L^2(0,T;L^2(\mathbb{R}^3_+))}\\ + K_B \Vert F_2 (\varphi) \Vert_{L^2(0,T;L^2(\mathbb{R}^3_-))} + K_S \Vert F_3 (\varphi) \Vert_{L^2(0,T;L^2(\mathbb{R}^2))}\\
\leq 2 \Vert L^{1/2} v_0 \Vert_H\\
 + \bigg\{ \frac{C_* K_S}{\alpha_S} + K_A C_* \bigg( \frac{1}{\sqrt{\beta}} + \frac{\kappa_A}{\tilde{\kappa}_S \sqrt{\beta}} \bigg) + K_B C_* \bigg( \frac{1}{\sqrt{\beta}} + \frac{\kappa_B}{\tilde{\kappa}_S \sqrt{\beta}} \bigg)  \bigg\} \Vert \varphi \Vert_{X_T}\\
+ T^{1/2} C (\kappa_A , \kappa_B , \tilde{\kappa}_S , \alpha_S , \beta ) \Vert \varphi \Vert_{X_T}. 
\end{multline}
By \eqref{eq46}, \eqref{eq47} and \eqref{eq48}, we see $(\rm{iii})$.
  \end{proof}

\begin{remark}\label{rem44}From property $(\rm{iii})$ of Lemma \ref{lem43}, we set $\alpha_0 = \alpha_0 (\tilde{\kappa}_S) = 8 C_* K_S$ and
\begin{equation*}
\beta_0 = \beta_0 (\kappa_A , \kappa_B , \tilde{\kappa}_S) = 64 C_*^2 \left\{ K_A \left( 1 + \frac{\kappa_A}{\tilde{\kappa}_S} \right) + K_B \left( 1 + \frac{\kappa_B}{\tilde{\kappa}_S} \right)  \right\}^2.
\end{equation*}
Then we see that if $\alpha_S > \alpha_0$, $\beta > \beta_0$, and $T \leq 1$ then 
\begin{equation*}
\Vert v \Vert_{X_T} \leq \Vert v_0 \Vert_H + 2 \Vert L^{1/2} v_0 \Vert_H + \frac{1}{4} \Vert \varphi \Vert_{X_T} + C_\star T^{1/2} \Vert \varphi \Vert_{X_T}.
\end{equation*}
Here $C_\star = C_\star ( \kappa_A , \kappa_B , \tilde{\kappa}_S , \alpha_S , \beta) >0$.
\end{remark}

\subsection{Uniqueness of Strong Solutions}\label{subsec42}

Let us study the uniqueness of the strong solutions to system \eqref{eq41}.
\begin{lemma}\label{lem45}
Let $v_0 \in H$, and let $v_\sharp = { }^t (v^\sharp_A , v^\sharp_B , v^\sharp_S), v_\flat = { }^t (v_A^\flat , v_B^\flat, v_S^\flat) \in X_T \cap L^2 ((0,T) \times \mathbb{R}^3_{+,-,0})$. Assume that $v_\sharp, v_\flat$ satisfy 
\begin{equation*}
\begin{cases}
\frac{d}{d t} v_\sharp + L v_\sharp = F (v_\sharp ) \text{ on }(0,T),\\
v_\sharp \vert_{t =0} = v_0,
\end{cases}{ \ }\begin{cases}
\frac{d}{d t} v_\flat + L v_\flat = F (v_\flat ) \text{ on }(0,T),\\
v_\flat \vert_{t =0} = v_0.
\end{cases}
\end{equation*}
Then $v_\sharp = v_\flat$ on $[0,T)$.
\end{lemma}

\begin{proof}[Proof of Lemma \ref{lem45}]
Set $v_* = v_* (t) = { }^t (v_A^* , v_B^* , v_S^*) := { }^t (v_A^\sharp - v_A^\flat, v_B^\sharp - v_B^\flat , v_S^\sharp - v_S^\flat)$. Then we see that $v_* \in X_T \cap L^2((0,T) \times \mathbb{R}^3_{+,-,0})$ and that $v_*$ satisfy
\begin{equation*}
\begin{cases}
\frac{d}{d t} v_* + L v_* = F (v_* ) \text{ on }(0,T),\\
v_* \vert_{t =0} = { }^t(0,0,0).
\end{cases}
\end{equation*}
Set
\begin{equation*}
\begin{cases}
\vartheta_A = \vartheta_A (x,t) := v_A^* (t) + v^*_S (t) {\rm{e}}^{- \beta x_3},\\
\vartheta_B = \vartheta_B (x,t) := v_B^* (t) + v^*_S (t) {\rm{e}}^{\beta x_3},\\
\vartheta_S = \vartheta_S ( x_h , t ) := v^*_S (t). 
\end{cases}
\end{equation*}
We easily check that
\footnotesize
\begin{align*}
\vartheta_A & \in C ([0, T); L^2 (\mathbb{R}^3_{+})) \cap L^2 (0,T; W^{2,2}(\mathbb{R}^3_+)) \cap W^{1,2}(0,T ; L^2(\mathbb{R}^3_+)) \cap L^2( (0,T) \times \mathbb{R}^3_{+}),\\
\vartheta_B & \in C ([0, T); L^2 (\mathbb{R}^3_{-})) \cap L^2 (0,T; W^{2,2}(\mathbb{R}^3_{-})) \cap W^{1,2}(0,T ; L^2(\mathbb{R}^3_{-})) \cap L^2((0,T) \times \mathbb{R}^3_{-}),\\
\vartheta_S & \in C ([0, T); L^2 (\mathbb{R}^2)) \cap L^2 (0,T; W^{2,2}(\mathbb{R}^2)) \cap W^{1,2}(0,T ; L^2(\mathbb{R}^2)) \cap L^2 ((0,T) \times \mathbb{R}^2),
\end{align*}\normalsize
and that
\begin{equation}\label{eq49}
\begin{cases}
\partial_t \vartheta_A = \kappa_A \Delta \vartheta_A & \text{ in } \mathbb{R}^3_+ \times (0, T) ,\\
\partial_t \vartheta_B = \kappa_B \Delta \vartheta_B & \text{ in } \mathbb{R}^3_{-} \times (0, T ),\\
\partial_t \vartheta_S = \tilde{\kappa}_S \Delta_h \vartheta_S + \frac{\kappa_A}{\alpha_S} \gamma_{+}[ \partial_3 \vartheta_A] - \frac{\kappa_B}{\alpha_S} \gamma_{-}[ \partial_3 \vartheta_B] & \text{ in } \mathbb{R}^2 \times (0, T ),\\
\gamma_{+} [\vartheta_A] = \gamma_{-} [\vartheta_B] = \vartheta_S & \text{ in } \mathbb{R}^2 \times (0, T),\\
\vartheta_A \vert_{t=0} = 0 \text{ in } \mathbb{R}^3_+, \vartheta_B \vert_{t = 0} = 0 \text{ in }\mathbb{R}^3_{-}, \vartheta_S \vert_{t=0} = 0 &\text{ in } \mathbb{R}^2 .
\end{cases}
\end{equation}
A direct calculation gives
\begin{multline*}
\frac{1}{2}\frac{d}{dt} ( \Vert \vartheta_A (t) \Vert_{L^2(\mathbb{R}^3_+)}^2 + \Vert \vartheta_B(t) \Vert_{L^2( \mathbb{R}^3_-)}^2 + \alpha_S \Vert \vartheta_S (t) \Vert_{L^2(\mathbb{R}^2)}^2)\\
= \left< \partial_t \vartheta_A , \vartheta_A  \right>_{L^2(\mathbb{R}^3_+)} + \left< \partial_t \vartheta_B, \vartheta_B \right>_{L^2 (\mathbb{R}^3_-)} + \left< \alpha_S \partial_t \vartheta_S , \vartheta_S \right>_{L^2(\mathbb{R}^2)}.
\end{multline*}
Using system \eqref{eq49} and integration by parts, we see that
\begin{multline*}
\frac{1}{2}\frac{d}{dt} ( \Vert \vartheta_A (t) \Vert_{L^2(\mathbb{R}^3_+)}^2 + \Vert \vartheta_B(t) \Vert_{L^2( \mathbb{R}^3_-)}^2 + \alpha_S \Vert \vartheta_S (t) \Vert_{L^2(\mathbb{R}^2)}^2)\\
= - \kappa_A \Vert \nabla \vartheta_A (t) \Vert_{L^2 (\mathbb{R}^3_+)}^2 - \kappa_B \Vert \nabla \vartheta_B (t) \Vert_{L^2(\mathbb{R}^3_-)}^2 - \tilde{\kappa}_S \alpha_S \Vert \nabla_h \theta_S (t) \Vert_{L^2(\mathbb{R}^2)}^2 .
\end{multline*}
Integrating with respect to time with $(\vartheta_A, \vartheta_B , \vartheta_S) \vert_{t=0} =(0,0,0)$, we find that for all $ t(0 < t < T)$
\begin{equation*}
\Vert \vartheta_A (t) \Vert_{L^2(\mathbb{R}^3_{+})}^2 + \Vert \vartheta_B (t) \Vert_{L^2(\mathbb{R}^3_{-})}^2 + \alpha_S \Vert \vartheta_S (t) \Vert_{L^2(\mathbb{R}^2)}^2 \leq 0. 
\end{equation*}
Therefore, we conclude that $(\vartheta_A, \vartheta_B , \vartheta_S) = (0,0,0)$ on $[0,T)$. This shows that $v^* =0$ and $v_\sharp = v_\flat$. Therefore, the lemma follows.
  \end{proof}

\subsection{Existence of a Local-in-Time Strong Solution}\label{subsec43}

In this section, we construct a local-in-time strong solution to system \eqref{eq41}. Let $\alpha_0$, $\beta_0$ be the two positive constants appearing in Remark \ref{rem44}.
\begin{proposition}[Existence of a local-in-time strong solution]\label{prop46}
Assume that $\alpha_S > \alpha_0$ and $\beta > \beta_0$. Then for each $v_0 = { }^t (v_0^A , v_0^B , v_0^S) \in D (L^{1/2})$, system \eqref{eq41} admits a unique local-in-time strong solution $v$ in 
\begin{equation*}
C ([0, T_*) ; H ) \cap L^2 (0, T_*; D(L)) \cap W^{1,2} (0, T_* ; H) \cap L^2((0,T_*) \times \mathbb{R}^3_{+,-,0} ),
\end{equation*}
satisfying
\begin{equation*}
\lim_{t \to 0 + 0} v (t) = v_0 \text{ in }H. 
\end{equation*}
Here $T_* = T_*(\kappa_A, \kappa_B,\tilde{\kappa}_S , \alpha_S, \beta) >0$ and $T_* \leq 1$.
\end{proposition}

\begin{proof}[Proof of Proposition \ref{prop46}]
Assume that $\alpha_S > \alpha_0$ and $\beta > \beta_0$, and that $T \leq 1$. Fix $v_0 = { }^t (v_0^A, v_0^B, v_0^S) \in D (L^{1/2})$. For each $m \in \mathbb{N}$, set $v_1 =v_1 (t)= { }^t (v_1^A, v_1^B, v_1^S)$ and $v_{m+1} = v_{m+1}(t) = { }^t (v_{m+1}^A , v_{m+1}^B , v_{m+1}^S)$ as follows:
\begin{equation}\label{eq4010}
\begin{cases}
\frac{d}{d t} v_1 + L v_1 = 0 \text{ on }(0,T),\\
v_1 \vert_{t = 0 } = v_0,
\end{cases}
\end{equation}
\begin{equation}\label{eq4011}
\begin{cases}
\frac{d}{d t} v_{m+1} + L v_{m+1} = F (v_m) \text{ on }(0,T),\\
v_{m + 1 } \vert_{t = 0 } = v_0.
\end{cases}
\end{equation}
We first consider system \eqref{eq4010}. Since $-L$ generates an analytic semigroup on $H$, we see that system \eqref{eq4010} admits a unique strong solution $v_1$ such that
\begin{align*}
& v_1 \in C ([0,T);H) \cap C((0,T);D(L)) \cap C^1((0,T);H),\\
& v_1 (t) = {\rm{e}}^{-t L} v_0 { \ }(0 < t < T),\\
& \lim_{t \to 0 + 0} v_1 (t) = v_0 \text{ in }H,\\
&  v_1 ,{ \ } d v_1/{dt},{ \ }Lv_1 \in C^{1/2}_{loc}((0,T);H).
\end{align*}
Moreover, from $({\rm{iii}})$ of Lemma \ref{lem32} we have
\begin{equation*}
\Vert v_1 \Vert_{X_T} \leq \Vert v_0 \Vert_H + 2 \Vert L^{1/2} v_0 \Vert_H.
\end{equation*}
By assertion $(\rm{ii})$ of Lemma \ref{lem42}, we see that
\begin{equation*}
F (v_1) \in L^2(0,T;H) \cap C^{1/2}_{loc}((0,T);H).
\end{equation*}

Next, we consider system \eqref{eq4011}. We prove that for each $m \in \mathbb{N}$ system \eqref{eq4011} admits a unique strong solution $v_{m+1}$ such that
\begin{equation}\label{eq4012}
v_{m+1} \in C ([0,T);H) \cap C((0,T);D(L)) \cap C^1((0,T);H),
\end{equation}
\begin{equation}
v_{m+1} (t) = {\rm{e}}^{ - t L} v_0 + \int_0^t {\rm{e}}^{- (t - \tau )L} F (v_m (\tau)) { \ }d \tau { \ }(0 < t <T),
\end{equation}
\begin{equation}
\lim_{t \to 0 + 0} v_{m+1} (t) = v_0 \text{ in }H,
\end{equation}
\begin{equation}
v_{m+1} ,{ \ } d v_{m+1}/{dt},{ \ }Lv_{m+1} \in C^{(\frac{1}{2})^{m+1}}_{loc}((0,T);H), 
\end{equation}
and
\begin{equation}\label{eq4016}
\Vert v_{m+1} \Vert_{X_T} \leq \Vert v_0 \Vert_H + 2 \Vert L^{1/2} v_0 \Vert_H + \left( \frac{1}{4} + C_{\star} T^{1/2} \right) \Vert v_m \Vert_{X_T}.
\end{equation}
Here $C_\star = C_\star (\kappa_A , \kappa_B , \tilde{\kappa}_S, \alpha_S , \beta )$ is the positive constant appearing in Remark \ref{rem44}.

We now consider the case when $m=1$, that is,
\begin{equation}\label{eq4017}
\begin{cases}
\frac{d}{d t} v_2 + L v_2 = F (v_1) \text{ on }(0,T),\\
v_2 \vert_{t = 0 } = v_0.
\end{cases}
\end{equation}
Since $v_0 \in H$, $F (v_1) \in L^2(0,T;H) \cap C^{1/2}_{loc}((0,T);H)$, $- L $ generates an analytic semigroup on $H$, we see that system \eqref{eq4017} admits a unique strong solution $v_2$ such that
\begin{align*}
& v_2 \in C ([0,T);H) \cap C((0,T);D(L)) \cap C^1((0,T);H),\\
& v_2 (t) = {\rm{e}}^{-t L} v_0 + \int_0^t {\rm{e}}^{- ( t - \tau) L} F (v_1 (\tau)){ \ }d \tau { \ }(0 < t < T).
\end{align*}
Since ${\rm{e}}^{- t L}$ is a $C_0$-semigroup on $H$ and $F (v_1) \in L^2(0,T;H)$, we find that
\begin{equation*}
\lim_{t \to 0 + 0} v_2 (t) = v_0 \text{ in }H.
\end{equation*}
Since $F (v_1) \in L^2(0,T;H) \cap C_{loc}^{1/2} ((0,T); H)$ and $v_0 \in D (L^{1/2})$, it follows from $(\rm{ii})$ of Lemma \ref{lem32} to see that
\begin{equation}\label{eq4018}
v_2 ,{ \ } d v_2/{dt},{ \ }Lv_2 \in C^{1/4}_{loc}((0,T);H).
\end{equation}
Using $(\rm{iii})$ of Lemma \ref{lem43} and Remark \ref{rem44}, we find that
\begin{equation}
\Vert v_2 \Vert_{X_T} \leq \Vert v_0 \Vert_H + 2 \Vert L^{1/2} v_0 \Vert_H + \left( \frac{1}{4} + C_{\star} T^{1/2} \right) \Vert v_1 \Vert_{X_T}.\label{eq4019}
\end{equation}
From \eqref{eq4018}, \eqref{eq4019}, and Lemma \ref{lem42}, we see that
\begin{equation*}
F (v_2) \in L^2(0,T;H) \cap C^{1/4}_{loc}((0,T);H).
\end{equation*}

We now consider the case when $m=2$, that is,
\begin{equation}\label{eq4020}
\begin{cases}
\frac{d}{d t} v_3 + L v_3 = F (v_2) \text{ on }(0,T),\\
v_3 \vert_{t = 0 } = v_0.
\end{cases}
\end{equation}
Since $v_0 \in H$, $F (v_2) \in L^2(0,T;H) \cap C^{1/4}_{loc}((0,T);H)$, $- L $ generates an analytic semigroup on $H$, we see that system \eqref{eq4017} admits a unique strong solution $v_3$ such that
\begin{align*}
& v_ 3\in C ([0,T);H) \cap C((0,T);D(L)) \cap C^1((0,T);H),\\
& v_3 (t) = {\rm{e}}^{-t L} v_0 + \int_0^t {\rm{e}}^{- ( t - \tau) L} F (v_2 (\tau)){ \ }d \tau { \ }(0 < t < T).
\end{align*}
Since ${\rm{e}}^{- t L}$ is a $C_0$-semigroup on $H$ and $F (v_2) \in L^2(0,T;H)$, we find that
\begin{equation*}
\lim_{t \to 0 + 0} v_3 (t) = v_0 \text{ in }H.
\end{equation*}
Since $F (v_2) \in L^2(0,T;H) \cap C_{loc}^{1/4} ((0,T); H)$ and $v_0 \in D (L^{1/2})$, it follows from $(\rm{ii})$ of Lemma \ref{lem32} to see that
\begin{equation}\label{eq4021}
v_3 ,{ \ } d v_3/{dt},{ \ }Lv_3 \in C^{1/8}_{loc}((0,T);H).
\end{equation}
Using $(\rm{iii})$ of Lemma \ref{lem43} and Remark \ref{rem44}, we find that
\begin{equation}\label{eq4022}
\Vert v_3 \Vert_{X_T} \leq \Vert v_0 \Vert_H + 2 \Vert L^{1/2} v_0 \Vert_H + \left( \frac{1}{4} + C_{\star} T^{1/2}\right) \Vert v_2 \Vert_{X_T}.
\end{equation}
From \eqref{eq4021}, \eqref{eq4022}, and Lemma \ref{lem42}, we see that
\begin{equation*}
F (v_3) \in L^2(0,T;H) \cap C^{1/8}_{loc}((0,T);H).
\end{equation*}

By induction, we see that for each $m \in \mathbb{N}$ system \eqref{eq4011} admits a unique strong solution $v_{m+1}$ satisfying \eqref{eq4012}-\eqref{eq4016}.

Next, we prove that for each $m \in \mathbb{N}$
\begin{equation}\label{eq4023}
\Vert v_{m+2} - v_{m+1} \Vert_{X_T} \leq \left( \frac{1}{4} + C_\star T^{1/2} \right) \Vert v_{m+1} - v_m \Vert_{X_T}.
\end{equation}
From
\begin{equation*}
\begin{cases}
\frac{d}{d t} v_{m+2} + L v_{m+2} = F (v_{m+1}) \text{ on }(0,T),\\
v_{m + 2 } \vert_{t = 0 } = v_0,
\end{cases}
\begin{cases}
\frac{d}{d t} v_{m+1} + L v_{m+1} = F (v_m),\\
v_{m + 1 } \vert_{t = 0 } = v_0,
\end{cases}
\end{equation*}
we have
\begin{equation*}
\begin{cases}
\frac{d}{d t} (v_{m+2} - v_{m+1}) + L (v_{m+2} - v_{m+1}) = F (v_{m+1} - v_m) \text{ on }(0,T),\\
(v_{m + 2 } - v_{m+1}) \vert_{t = 0 } = 0.
\end{cases}
\end{equation*}
By $(\rm{iii})$ of Lemma \ref{lem43} and Remark \ref{rem44}, we have \eqref{eq4023}.

Now we choose $T_*$ such that $C_\star T_*^{1/2} \leq 1/4$. Then we have
\begin{equation*}
\Vert v_{m+2} - v_{m+1} \Vert_{X_{T_*}} \leq \frac{1}{2} \Vert v_{m+1} - v_m \Vert_{X_{T_*}}.
\end{equation*}
From \eqref{eq4016}, we have
\begin{equation*}
\Vert v_{m+1} \Vert_{X_{T_*}} \leq \Vert v_0 \Vert_H + 2 \Vert L^{1/2} v_0 \Vert_H + \frac{1}{2} \Vert v_m \Vert_{X_{T_*}}.
\end{equation*}
Since $\Vert v_1 \Vert_{X_{T_*}} \leq \Vert v_0 \Vert_H + 2 \Vert L^{1/2} v_0 \Vert_H$, we see that for each $m \in \mathbb{N}$
\begin{equation*}
\Vert v_{m+1} \Vert_{X_{T_*}} \leq 2 \Vert v_0 \Vert_H + 4 \Vert L^{1/2} v_0 \Vert_H.
\end{equation*}
We also see that
\begin{align*}
\Vert v_{m + 2} - v_{m+1} \Vert_{X_{T_*}} & \leq \bigg(\frac{1}{2}\bigg)^{m} \Vert v_2 - v_1 \Vert_{X_{T_*}},\\
 & \leq \bigg( \frac{1}{2} \bigg)^{m} ( 3 \Vert v_0 \Vert_H + 6 \Vert L^{1/2} v_0 \Vert_H ) .
\end{align*}
From a fixed-point argument, we have a unique function $v$ in $X_{T_*}$ satisfying
\begin{align}
& \lim_{m \to \infty} \Vert v_m - v \Vert_{X_{T_*}} = 0,\label{eq4024}\\
& \Vert v \Vert_{X_{T_*}} \leq 2 \Vert v_0 \Vert_H + 4 \Vert L^{1/2} v_0 \Vert_H. \notag
\end{align}
From $(\rm{i})$ of Lemma \ref{lem42}, we see that
\begin{equation}\label{eq4025}
F (v) \in L^2(0,T_* ;H).
\end{equation}
Since $v_{m+1}$ satisfies
\begin{equation*}
\begin{cases}
\frac{d}{dt} v_{m+1} + L v_{m+1} = F (v_m) \text{ on }(0,T_*),\\
v_{m+1}\vert_{t = 0} = v_0,
\end{cases}
\end{equation*}
and for $0< t < T_*$
\begin{equation*}
v_{m+1} (t) = {\rm{e}}^{- t L}v_0 + \int_0^t {\rm{e}}^{- (t - \tau )L} F (v_m(\tau)) { \ }d \tau,
\end{equation*}
we apply \eqref{eq4024} and Lemmas \ref{lem22}, \ref{lem42} to see that
\begin{equation*}
\left\Vert \frac{d}{dt}v + L v - F (v) \right\Vert_{L^2(0,T_*;H)} = 0
\end{equation*}
and that for $0< t <T_*$
\begin{equation*}
v (t) = {\rm{e}}^{- t L}v_0 + \int_0^t {\rm{e}}^{- (t - \tau )L} F (v(\tau)) { \ }d \tau.
\end{equation*}
Since ${\rm{e}}^{- t L}$ is a $C_0$-semigroup on $H$, we use the Cauchy-Schwarz inequality and \eqref{eq4025} to observe that
\begin{align*}
\Vert v(t) - v_0 \Vert_H & \leq \Vert {\rm{e}}^{- t L} v_0 - v_0 \Vert_H + t^{1/2} \Vert F (v) \Vert_{L^2(0,T;H)}\\
& \to 0 { \ }(t \to 0 + 0).
\end{align*}

Now we show that
\begin{equation}\label{eq4026}
v \in L^2((0,T_*) \times \mathbb{R}^3_{+,-,0}).
\end{equation}
From Lemma \ref{lem31}, we consider $v_1$ and $v_{m+1}$ as follows:
\begin{align*}
v_1 (x,t) & = G* v_0,\\
v_{m+1} (x,t) &= G*v_0 + \int_0^t G*F (v_m(x,\tau)) { \ }d \tau. 
\end{align*}
Here $G$ denotes the heat kernels (see Lemmas \ref{lem31} and \ref{lem21} for $G$). We easily check that for each $m \in \mathbb{N}$, $v_m \in L^2((0,T_*) \times \mathbb{R}^3_{+,-,0}) $. By \eqref{eq4020}, we see that
\begin{equation*}
\Vert v_{m+1} - v_m \Vert_{L^2((0,T_*) \times \mathbb{R}^3_+)}  \to 0  \text{ as }m \to \infty.
\end{equation*}
Therefore, we see \eqref{eq4026}. From Lemma \ref{lem45}(the uniqueness of the strong solutions to system \eqref{eq41}), we see that $v$ is a unique local-in-time strong solution to \eqref{eq41}. Therefore, Proposition \ref{prop46} is proved.
  \end{proof}

\subsection{Existence of a Global-in-Time Strong Solution}\label{subsec44}

In this section, we construct a global-in-time strong solution of system \eqref{eq41}. Let $\alpha_0$, $\beta_0$ be the two positive constants appearing in Remark \ref{rem44}, and let $T_*$ the positive constant appearing in Proposition \ref{prop46}. 
\begin{proof}[Proof of Theorem \ref{thm41}]
Assume that $\alpha_S > \alpha_0$ and $\beta > \beta_0$. Fix $v_0 \in D (L^{1/2})$. We first consider the following system:
\begin{equation}\label{eq4027}
\begin{cases}
\frac{d}{dt}v^1 + L v^1 = F (v^1) \text{ on } (0,T_*),\\
v^1 \vert_{t=0} = v_0.
\end{cases}
\end{equation}
Since $T_*$ does not depend on initial data, we apply Proposition \ref{prop46} to see that there exists a unique strong solution $v^1$ of system \eqref{eq4027}. Next, we consider the following system:
\begin{equation}\label{eq4028}
\begin{cases}
\frac{d}{dt}v^2 + L v^2 = F (v^2) \text{ on } (T_*/2,T_*/2 + T_*),\\
v^2 \vert_{t=T_*/2} = v^1(T_*/2).
\end{cases}
\end{equation}
Since $T_*$ does not depend on initial data, it follows from Proposition \ref{prop46} to see that there exists a unique strong solution $v^2$ of system \eqref{eq4028}. Next, we consider the following system:
\begin{equation}\label{eq4029}
\begin{cases}
\frac{d}{dt}v^3 + L v^3 = F (v^3) \text{ on } (T_*, 2T_*),\\
v^3 \vert_{t = T_*} = v^2 (T_*).
\end{cases}
\end{equation}
From Proposition \ref{prop46}, there exists a unique strong solution $v^3$ of system \eqref{eq4029}. Now we set
\begin{equation*}
v = v (t) =
\begin{cases}
v^1 &\text{ on }(0,3T_*/4],\\
v^2 & \text{ on }(3 T_*/4, 5T_*/4],\\
v^3 & \text{ on }(5T_*/4, 2 T_*).
\end{cases}
\end{equation*}
Then we find that
\begin{equation*}
v \in C([0,2 T_*); H) \cap L^2(0,2 T_*; D (L) ) \cap W^{1,2} (0, 2 T_* ; H) \cap L^2((0,2T_*) \times \mathbb{R}^3_{+,-,0}),
\end{equation*}
and that
\begin{equation*}
\begin{cases}
\frac{d}{dt}v + L v = F (v) \text{ on } (0,2 T_*),\\
v \vert_{t=0} = v_0.
\end{cases}
\end{equation*}
Let $T >1$. Since $T_*$ does not depend on initial data, we repeat the same argument above to see that there is a function $v$ such that
\begin{equation*}
v \in C([0, T); H) \cap L^2(0,T; D (L) ) \cap W^{1,2} (0, T ; H) \cap L^2((0,T) \times \mathbb{R}^3_{+,-,0})
\end{equation*}
and
\begin{equation*}
\begin{cases}
\frac{d}{dt}v + L v = F (v) \text{ on } (0, T),\\
v \vert_{t=0} = v_0.
\end{cases}
\end{equation*}
Since we can choose $T$ to be any positive number, we find that there is $v$ such that
\begin{equation*}
v \in C([0, \infty); H) \cap L_{loc}^2(0, \infty; D (L) ) \cap W_{loc}^{1,2} (0, \infty ; H) \cap L^2_{loc}(\mathbb{R}_+ \times \mathbb{R}^3_{+,-,0})
\end{equation*}
and
\begin{equation*}
\begin{cases}
\frac{d}{dt}v + L v = F (v) \text{ on } (0, \infty),\\
v \vert_{t=0} = v_0.
\end{cases}
\end{equation*}
From Lemma \ref{lem45}, we see that $v$ is a unique global-in-time strong solution of system \eqref{eq41} with initial data $v_0$. Therefore, Theorem \ref{thm41} is proved.
  \end{proof}


\section{Energy Equality}\label{sect5}

In this section, we prove Theorem \ref{thm11}. We first study the uniqueness of the strong solutions to system \eqref{eq11}. Then we apply Theorem \ref{thm41} to show the existence of a global-in-time strong solution to \eqref{eq11}. Finally, we construct an energy equality of our heat system.

\begin{lemma}[Uniqueness]\label{lem51}Let $\kappa_A$, $\kappa_B$, $\tilde{\kappa}_S$, $\alpha_S >0$. Let $\theta_0^A \in W^{1,2} (\mathbb{R}^3_{+})$, $\theta_0^B \in W^{1,2}( \mathbb{R}^3_{-})$, $\theta_0^S \in W^{1,2} (\mathbb{R}^2)$ satisfying $\gamma_{+}[\theta_0^A] = \theta_0^S$ and $\gamma_{-}[\theta_0^B] = \theta_0^S$. Let $T>0$. Let
\footnotesize
\begin{align*}
\theta^\sharp_A, \theta^\flat_A & \in C ([0, \infty); L^2 (\mathbb{R}^3_{+})) \cap L^2 (0, T; W^{2,2} (\mathbb{R}^3_+) ) \cap W^{1,2}(0,T ; L^2 (\mathbb{R}^3_+) ) \cap L^2( (0,T) \times \mathbb{R}^3_{+} ),\\
\theta_B^\sharp, \theta_B^\flat & \in C ([0, \infty); L^2 (\mathbb{R}^3_{-})) \cap L^2 (0, T; W^{2,2} (\mathbb{R}^3_{-}) ) \cap W^{1,2}(0,T ; L^2 (\mathbb{R}^3_{-}) ) \cap L^2 ((0,T) \times \mathbb{R}^3_{-} ),\\
\theta_S^\sharp, \theta_S^\flat & \in C ([0, \infty); L^2 (\mathbb{R}^2)) \cap L^2 (0,T; W^{2,2} (\mathbb{R}^2) ) \cap W^{1,2}(0, T ; L^2 (\mathbb{R}^2) ) \cap L^2((0,T) \times \mathbb{R}^2 ).
\end{align*}\normalsize
Assume that $(\theta_A^\sharp , \theta_B^\sharp , \theta_S^\sharp)$ and $(\theta_A^\flat , \theta_B^\flat , \theta_S^\flat )$ are two strong solution of system \eqref{eq11} with initial data $(\theta_0^A , \theta_0^B, \theta_0^S)$. Then $\theta_A^\sharp = \theta_A^\flat$, $\theta_B^\sharp = \theta_B^\flat$, and $\theta_S^\sharp = \theta_S^\flat$ \text{ on }$[0,T)$.
\end{lemma}

\begin{proof}[Proof of Lemma \ref{lem51}]
Set $(\theta_A^*, \theta_B^*, \theta_S^*) = (\theta_A^\sharp - \theta_A^\flat , \theta_B^\sharp - \theta_B^\flat , \theta_S^\sharp - \theta_S^\flat)$. Then we have
\begin{equation*}
\begin{cases}
\partial_t \theta^*_A = \kappa_A \Delta \theta^*_A & \text{ in } \mathbb{R}^3_+ \times (0, T) ,\\
\partial_t \theta^*_B = \kappa_B \Delta \theta^*_B & \text{ in } \mathbb{R}^3_{-} \times (0, T ),\\
\alpha_S \partial_t \theta^*_S = \tilde{\kappa}_S \alpha_S \Delta_h \theta^*_S + \kappa_A \gamma_{+}[ \partial_3 \theta^*_A] - \kappa_B \gamma_{-}[ \partial_3 \theta^*_B] & \text{ in } \mathbb{R}^2 \times (0, T ),\\
\gamma_{+} [ \theta^*_A] = \gamma_{-} [\theta^*_B] = \theta^*_S & \text{ in } \mathbb{R}^2 \times (0, T),\\
\theta^*_A \vert_{t=0} = 0{ \ } \text{ in } \mathbb{R}^3_+, { \ }\theta^*_B \vert_{t = 0} = 0{ \ } \text{ in }\mathbb{R}^3_{-}, { \ }\theta^*_S \vert_{t=0} = 0{ \ } \text{ in } \mathbb{R}^2.
\end{cases}
\end{equation*}
By an argument similar to prove Lemma \ref{lem45}, we see that for all $0 \leq t <T$
\begin{equation*}
\Vert \theta^*_A (t) \Vert^2_{L^2 (\mathbb{R}^3_{+})} + \Vert \theta^*_B (t) \Vert^2_{L^2 (\mathbb{R}^3_{-})} + \alpha_S \Vert \theta^*_S (t) \Vert^2_{L^2 (\mathbb{R}^2)} \leq 0.
\end{equation*}
Therefore, we see that $(\theta_A^*, \theta_B^* , \theta_S^* ) = (0,0,0 )$ on $[0,T)$, that is, $\theta^\sharp = \theta^\flat$.
  \end{proof}

Let us prove Theorem \ref{thm11}.
\begin{proof}[Proof of Theorem \ref{thm11}]
Let $\kappa_A$, $\kappa_B$, $\tilde{\kappa}_S$, $\alpha_S$, $\beta >0$. Assume that $\alpha_S > \alpha_0$ and $\beta > \beta_0$. Let $\theta_0^A \in W^{1,2} (\mathbb{R}^3_{+})$, $\theta_0^B \in W^{1,2}( \mathbb{R}^3_{-})$, $\theta_0^S \in W^{1,2} (\mathbb{R}^2)$ satisfying $\gamma_{+}[\theta_0^A] = \theta_0^S$ and $\gamma_{-}[\theta_0^B] = \theta_0^S$.  Set $v_0^A = v_0^A (x) := \theta_0^A - \theta_0^S {\rm{e}}^{- \beta x_3} $, $v_0^B = v_0^B (x) := \theta_0^A - \theta_0^S {\rm{e}}^{\beta x_3}$, and $v_0^S = v_0^S (x_h) = \theta_0^S$. It is clear that $v_0^A \in W_0^{1,2} ( \mathbb{R}^3_+)$, $v_0^B \in W^{1,2}_0 ( \mathbb{R}^3_-)$, and $v_0^S \in W^{1,2} (\mathbb{R}^2)$. Set $v_0 = { }^t (v_0^A , v_0^B ,v_0^S) $. Since $v_0 \in D (L^{1/2})$, it follows from Theorem \ref{thm41} to see that system \eqref{eq41} with initial data $v_0$ admits a unique global-in-time strong solution $v= { }^t (v_A , v_B,v_S)$ in
\begin{equation*}
C ([0,\infty) ; H ) \cap L^2_{loc} (0,\infty; D(L)) \cap W_{loc}^{1,2} (0,\infty ; H) \cap L^2_{loc} ( \mathbb{R}_+ \times \mathbb{R}^3_{+,-,0}),
\end{equation*}
satisfying
\begin{equation}\label{eq51}
\lim_{t \to 0 + 0} v (t) = v_0 \text{ in }H. 
\end{equation}
Set
\begin{equation*}
\begin{cases}
\theta_A = \theta_A (x,t) := v_A (t) + v_S (t) {\rm{e}}^{- \beta x_3},\\
\theta_B = \theta_B (x,t) := v_B (t) + v_S (t) {\rm{e}}^{\beta x_3},\\
\theta_S = \theta_S (x,t) := v_S (t). 
\end{cases}
\end{equation*}
We easily check that
\footnotesize
\begin{align*}
\theta_A & \in C ([0, \infty); L^2 (\mathbb{R}^3_{+})) \cap L_{loc}^2 (0,\infty; W^{2,2}(\mathbb{R}^3_+)) \cap W_{loc}^{1,2}(0,\infty ; L^2(\mathbb{R}^3_+)) \cap L^2_{loc}( \mathbb{R}_+ \times \mathbb{R}^3_{+}),\\
\theta_B & \in C ([0, \infty); L^2 (\mathbb{R}^3_{-})) \cap L_{loc}^2 (0,\infty; W^{2,2}(\mathbb{R}^3_{-})) \cap W_{loc}^{1,2}(0,\infty ; L^2(\mathbb{R}^3_{-})) \cap L^2_{loc}(\mathbb{R}_+ \times \mathbb{R}^3_{-}),\\
\theta_S & \in C ([0, \infty); L^2 (\mathbb{R}^2)) \cap L_{loc}^2 (0,\infty; W^{2,2}(\mathbb{R}^2)) \cap W_{loc}^{1,2}(0,\infty ; L^2(\mathbb{R}^2)) \cap L_{loc}^2 (\mathbb{R}_+ \times \mathbb{R}^2),
\end{align*}\normalsize 
and $(\theta_A , \theta_B, \theta_S)$ satisfy \eqref{eq11}.

Now we prove that
\begin{align}
\lim_{ t \to 0 +0} \Vert \theta_A (t) - \theta_0^A \Vert_{L^2 ( \mathbb{R}^3_{+})} = 0, \label{eq52}\\
\lim_{ t \to 0 +0} \Vert \theta_B (t) - \theta_0^B \Vert_{L^2 ( \mathbb{R}^3_{-})} = 0, \label{eq53}\\
\lim_{ t \to 0 +0} \Vert \theta_S (t) - \theta_0^S \Vert_{L^2 ( \mathbb{R}^2)} = 0. \label{eq54}
\end{align}
From \eqref{eq51}, we have
\begin{equation*}
\lim_{t \to 0 + 0} ( \Vert v_A (t) - v^A_0 \Vert_{L^2(\mathbb{R}^3_+)} + \Vert v_B (t) - v_0^B \Vert_{L^2(\mathbb{R}^3_-)} + \Vert v_S (t) - v_0^S \Vert_{L^2(\mathbb{R}^2)}  ) = 0.
\end{equation*}
This gives \eqref{eq54}. Since $\theta_A = v_A + v_S {\rm{e}}^{- \beta x_3}$, $\theta_S = v_S$, $v_0^A = \theta_0^A - \theta_0^A {\rm{e}}^{- \beta x_3}$, and $v_0^S = \theta_0^S$, we use Lemma \ref{lem24} to check that
\begin{align*}
\Vert \theta_A (t) - \theta_0^A \Vert_{L^2 ( \mathbb{R}^3_{+})} = & \Vert v_A (t) + v_S(t) {\rm{e}}^{- \beta x_3} - (\theta_0^A - \theta_0^S {\rm{e}}^{- \beta x_3}) + \theta_0^S {\rm{e}}^{- \beta x_3} \Vert_{L^2 (\mathbb{R}^3_{+})}\\
= & \Vert v_A (t) - v_0^A \Vert_{L^2 ( \mathbb{R}^3_{+})} + \Vert v_S (t) {\rm{e}}^{- \beta x_3} - v_0^S {\rm{e}}^{- \beta x_3} \Vert_{L^2 (\mathbb{R}^3_{+})}\\
\leq & \Vert v_A (t) - v_0^A \Vert_{L^2 ( \mathbb{R}^3_{+})} + \frac{1}{\sqrt{2 \beta}} \Vert \theta_S (t) - \theta_0^S \Vert_{L^2 (\mathbb{R}^2)}\\
\leq & C \Vert v (t) - v_0 \Vert_H \to 0 \text{ as } t \to 0 + 0 .
\end{align*}
Therefore, we have \eqref{eq52}. Similarly, we see \eqref{eq53}.

Finally, we construct an energy equality for our system. By an argument similar to prove Lemma \ref{lem45}, we see that for all $0 \leq t_1 \leq t_2 < \infty$, \eqref{eq13} holds. Therefore, Theorem \ref{thm11} is proved.
  \end{proof}


\section{Appendix (I): Derivation of Heat Equations in $\mathbb{R}^3_+$, $\mathbb{R}^3_-$ with $\mathbb{R}^3_0$}\label{sect6}
In Appendix (I), we derive the heat equations \eqref{eq11} in the two half spaces $\mathbb{R}^3_+$, $\mathbb{R}^3_-$, and the interface $\mathbb{R}^2 \times \{ 0 \} (\cong \mathbb{R}^2)$ by applying a method in \cite{K23}. We consider the temperatures of a simple ocean-atmosphere model with an interface or oil floating on water (see Figure \ref{Fig1}).

Let us first introduce our settings. Let $T \in (0, \infty]$. Define $\Omega_{A,T} = \mathbb{R}^3_{+} \times (0,T)$, $\Omega_{B,T} = \mathbb{R}^3_{-} \times (0,T) $, $\overline{\Omega}_{A,T} = \overline{ \mathbb{R}^3_{+}} \times (0,T)$, $\overline{\Omega}_{B,T} = \overline{ \mathbb{R}^3_{-}} \times (0,T)$, and $\Omega_{S,T} = \mathbb{R}^2 \times (0,T)$. For $\sharp = A , B$, let $\rho_\sharp = \rho_\sharp (x,t)$, $v_\sharp = v_\sharp (x,t) = { }^t (v^\sharp_1, v^\sharp_2 , v^\sharp_3)$, $\theta_\sharp = \theta_\sharp (x,t)$, $\kappa_\sharp = \kappa_\sharp (x,t)$, and $\mathcal{C}_\sharp = \mathcal{C}_\sharp (x,t)$ be the \emph{density}, the \emph{velocity}, the \emph{temperature}, the \emph{thermal conductivity}, and the \emph{specific heat} of the fluid in $\Omega_{\sharp, T}$, respectively. Let $\rho_S = \rho_S (x_h,t)$, $v_S = v_S (x_h,t) = { }^t (v^S_1, v^S_2)$, $\theta_S = \theta_S (x_h,t)$, $\kappa_S = \kappa_S (x_h,t)$, and $\mathcal{C}_S = \mathcal{C}_S (x_h,t)$ be the \emph{density}, the \emph{velocity}, the \emph{temperature}, the \emph{thermal conductivity}, and the \emph{specific heat} of the fluid in $\mathbb{R}^3_0 \times (0,T) (\cong \Omega_{S, T})$, respectively. Define
\begin{align*}
& C_0^\infty ( \overline{ \mathbb{R}^3_{+} }) := \{ f \in C^\infty (\overline{ \mathbb{R}^3_{+}  }); { \ }\lim_{x \in \overline{\mathbb{R}^3_{+}},\vert x \vert \to \infty} \vert f (x) \vert = 0 \},\\
& C_0^\infty ( \overline{ \mathbb{R}^3_{-} }) := \{ f \in C^\infty (\overline{ \mathbb{R}^3_{-}  }); { \ }\lim_{x \in \overline{\mathbb{R}^3_{-}},\vert x \vert \to \infty} \vert f (x) \vert = 0 \},\\
& C_0^\infty ( \overline{\Omega}_{A,T}) := \{ f \in C^\infty (\overline{\Omega }_{A,T}); { \ }\lim_{x \in \overline{\mathbb{R}^3_{+}},\vert x \vert \to \infty} \vert f (x,t) \vert = 0 \text{ for each } 0 < t <T \},\\
& C_0^\infty ( \overline{ \Omega }_{B,T}) := \{ f \in C^\infty (\overline{ \Omega }_{B,T}); { \ }\lim_{x \in \overline{\mathbb{R}^3_{-}},\vert x \vert \to \infty} \vert f (x,t) \vert = 0 \text{ for each } 0 < t <T \},\\
& C_0^\infty ( \Omega_{S,T} ) := \{ f \in C^\infty (\Omega_{S,T}); { \ }\lim_{x_h \in \mathbb{R}^2,\vert x_h \vert \to \infty} \vert f (x_h, t) \vert = 0 \text{ for each } 0 < t <T \}.
\end{align*}
We assume that $\rho_A$, $\mathcal{C}_A \in C^\infty (\overline{\Omega}_{A,T})$, $v_1^A$, $v_2^A$, $v_3^A$, $\theta_A \in C_0^\infty ( \overline{\Omega}_{A,T})$, $\rho_B$, $\mathcal{C}_B \in C^\infty (\overline{\Omega}_{B,T})$, $v_1^B$, $v_2^B$, $v_3^B$, $\theta_B \in C_0^\infty ( \overline{ \Omega }_{B,T})$, $\rho_S$, $\mathcal{C}_S \in C^\infty (\Omega_{S,T})$, $v_1^S$, $v_2^S$, $\theta_S \in C_0^\infty ( \Omega_{S,T})$, and that $\kappa_A$, $\kappa_B$, $\kappa_S$ are three positive constants.

\begin{definition}[Velocity fields, Transport theorems]\label{def61}{ \ }\\ We say that $(\Omega_{A,T}, \Omega_{B,T}, \Omega_{S,T})$ is \emph{flowed by the velocity fields} $(v_A,v_B ,v_S)$ if for each $0< t <T$, $\varphi_A \in C^1 (\Omega_{A,T})$, $\varphi_B \in C^1(\Omega_{B,T})$, $\varphi_S \in C^1 (\Omega_{S,T})$, and $\Lambda \subset \mathbb{R}^3$,
\begin{align*}
 \frac{d}{d t} \int_{\mathbb{R}^3_{+} \cap \Lambda} \varphi_A (x,t) { \ }d x & = \int_{\mathbb{R}^3_{+} \cap \Lambda}\{ \partial_t \varphi_A + (v_A \cdot \nabla) \varphi_A + (\nabla \cdot v_A ) \varphi_A \} { \ }dx,\\ 
 \frac{d}{d t} \int_{\mathbb{R}^3_{-} \cap \Lambda} \varphi_B (x,t) { \ }d x & = \int_{\mathbb{R}^3_{-} \cap \Lambda} \{ \partial_t \varphi_B + (v_B \cdot \nabla) \varphi_B + (\nabla \cdot v_B ) \varphi_B \} { \ }dx,\\ 
 \frac{d}{d t} \int_{\mathbb{R}^2 \cap \Lambda} \varphi_S (x_h,t) { \ }d x_h & = \int_{\mathbb{R}^2 \cap \Lambda} \{ \partial_t \varphi_S + (v_S \cdot \nabla_h) \varphi_S + (\nabla_h \cdot v_S ) \varphi_S \} { \ }d x_h. 
\end{align*}
\noindent We often call the above three equalities the \emph{transport theorems}.
\end{definition}
Throughout Appendix (I), we assume that $(\Omega_{A,T}, \Omega_{B,T}, \Omega_{S,T})$ is \emph{flowed by the velocity fields} $(v_A,v_B ,v_S)$. From the transport theorems, we admit that the densities $(\rho_A, \rho_B, \rho_S)$ satisfies
\begin{equation}\label{eq61}
\begin{cases}
\partial_t \rho_A + (v_A \cdot \nabla) \rho_A + (\nabla \cdot v_A ) \rho_A = 0 & \text{ in } \mathbb{R}^3_{+} \times (0, T),\\
\partial_t \rho_B + (v_B \cdot \nabla) \rho_B + (\nabla \cdot v_B ) \rho_B = 0 & \text{ in } \mathbb{R}^3_{-} \times (0, T),\\
\partial_t \rho_S + (v_S \cdot \nabla_h) \rho_S + (\nabla_h \cdot v_S ) \rho_S = 0 & \text{ in } \mathbb{R}^2 \times (0, T).
\end{cases}
\end{equation}

Next, we consider the variation of energies dissipation due to thermal diffusion. Set
\begin{equation*}
E_{TD}[\theta_A , \theta_B , \theta_S] = - \int_{\mathbb{R}^3_{+}} \frac{\kappa_A}{2} \vert \nabla \theta_A \vert^2 d x - \int_{\mathbb{R}^3_{-}} \frac{\kappa_B}{2} \vert \nabla \theta_B \vert^2 d x - \int_{\mathbb{R}^2} \frac{\kappa_S}{2} \vert \nabla_h  \theta_S \vert^2 d x_h.
\end{equation*}
Under the restriction that
\begin{equation}\label{eq62}
\theta_A \vert_{x_3=0} = \theta_B \vert_{x_3 =0} = \theta_S \text{ in }\mathbb{R}^2 \times (0,T),
\end{equation}
we consider the variation of the dissipation energies $E_{TD}$. Fix $t$. For $- 1 < \varepsilon <1$, $\phi_A \in C_0^\infty (\overline{ \mathbb{R}^3_+ })$, $\phi_B \in C_0^\infty ( \overline{ \mathbb{R}^3_{-} } )$, $\phi_S \in C_0^\infty ( \mathbb{R}^2)$, set $\theta_A^\varepsilon = \theta_A + \varepsilon \phi_A$, $\theta_B^\varepsilon = \theta_B + \varepsilon \theta_B$, $\theta_S^\varepsilon = \theta_S + \varepsilon \phi_S$. From \eqref{eq62}, we assume that for $- 1 < \varepsilon < 1$
\begin{equation*}
\theta^\varepsilon_A \vert_{x_3=0} = \theta^\varepsilon_B \vert_{x_3 =0} = \theta^\varepsilon_S \text{ in } \mathbb{R}^2.
\end{equation*}
Then we have
\begin{equation}\label{eq63}
\phi_A \vert_{x_3=0} = \phi_B \vert_{x_3 =0} = \phi_S \text{ in }\mathbb{R}^2.
\end{equation}
A direct calculation gives
\begin{multline*}
\frac{d}{d \varepsilon} \bigg\vert_{\varepsilon = 0} E_{TD}[\theta_A^\varepsilon , \theta_B^\varepsilon , \theta_S^\varepsilon] = - \int_{\mathbb{R}^3_{+}} \kappa_A \nabla \theta_A \cdot \nabla \phi_A { \ }d x - \int_{\mathbb{R}^3_{-}} \kappa_B \nabla \theta_B \cdot \nabla \phi_B { \ }d x\\ - \int_{\mathbb{R}^2} \kappa_S \nabla_h  \theta_S \cdot \nabla_h \phi_S { \ }d x_h.
\end{multline*}
Using integration by parts with \eqref{eq63}, we see that
\begin{multline*}
\frac{d}{d \varepsilon} \bigg\vert_{\varepsilon = 0} E_{TD}[\theta_A^\varepsilon , \theta_B^\varepsilon , \theta_S^\varepsilon] = \int_{\mathbb{R}^3_{+}} (\kappa_A \Delta \theta_A ) \phi_A { \ }d x + \int_{\mathbb{R}^3_{-}} ( \kappa_B \Delta \theta_B ) \phi_B {  \ } d x\\ + \int_{\mathbb{R}^2} ( \kappa_S \Delta_h \theta_S + \kappa_A \partial_3 \theta_A\vert_{x_3 =0} - \kappa_B \partial_3 \theta_B\vert_{x_3 =0} ) \phi_S { \ }d x_h.
\end{multline*}
Therefore, we have the following proposition.
\begin{proposition}\label{prop62}
Fix $t$. Let $q_A, q_B,q_S \in C (\mathbb{R}^3)$. Assume that for every $\phi_A \in C_0^\infty (\overline{\mathbb{R}^3_+})$, $\phi_B \in C_0^\infty (\overline{ \mathbb{R}^3_-})$, $\phi_S \in C_0^\infty (\mathbb{R}^2)$ satisfying \eqref{eq63},
\begin{equation*}
\frac{d}{d \varepsilon} \bigg\vert_{\varepsilon = 0} E_{TD}[\theta_A^\varepsilon , \theta_B^\varepsilon , \theta_S^\varepsilon] = \int_{\mathbb{R}^3_{+}} q_A \phi_A { \ }d x + \int_{\mathbb{R}^3_{-}} q_B \phi_B { \ }d x + \int_{\mathbb{R}^2} q_S \phi_S { \ }d x_h.
\end{equation*}
Then
\begin{equation*}
\begin{cases}
q_A = \kappa_A \Delta \theta_A & \text{ in }\mathbb{R}^3_{+},\\
q_B = \kappa_B \Delta \theta_B & \text{ in } \mathbb{R}^3_{-},\\
q_S = \kappa_S \Delta_h \theta_S + \kappa_A \partial_3 \theta_A \vert_{x_3 = 0} - \kappa_B \partial_3 \theta_B \vert_{x_3 =0} & \text{ in } \mathbb{R}^2.
\end{cases}
\end{equation*}
\end{proposition}

\begin{proof}[Proof of Proposition \ref{prop62}]
Fix $t$. Let $q_A, q_B,q_S \in C (\mathbb{R}^3)$. We first consider the case when $\phi_S \equiv 0$ in $\mathbb{R}^2$. By assumption, we see that for all $\phi_A \in C_0^\infty (\mathbb{R}^3_+)$, $\phi_B \in C_0^\infty ( \mathbb{R}^3_-)$,
\begin{multline*}
\int_{\mathbb{R}^3_{+}} (\kappa_A \Delta \theta_A ) \phi_A { \ }d x + \int_{\mathbb{R}^3_{-}} ( \kappa_B \Delta \theta_B ) \phi_B {  \ } d x = \int_{\mathbb{R}^3_{+}} q_A \phi_A { \ }d x + \int_{\mathbb{R}^3_{-}} q_B \phi_B { \ }d x.
\end{multline*}
This implies that $q_A = \kappa_A \Delta \theta_A$ in $\mathbb{R}^3_+$ and $q_B = \kappa_B \Delta \theta_B$ in $\mathbb{R}^3_-$.

Next we consider the case when $\phi_S \in C_0^\infty (\mathbb{R}^2)$. Let $\phi_A \in C_0^\infty (\overline{\mathbb{R}^3_+})$, $\phi_B \in C_0^\infty (\overline{ \mathbb{R}^3_-})$, $\phi_S \in C_0^\infty (\mathbb{R}^2)$ satisfying \eqref{eq63}. Since $q_A = \kappa_A \Delta \theta_A$ in $\mathbb{R}^3_+$ and $q_B = \kappa_B \Delta \theta_B$ in $\mathbb{R}^3_-$, we find that 
\begin{equation*}
 \int_{\mathbb{R}^2} ( \kappa_S \Delta_h \theta_S + \kappa_A \partial_3 \theta_A\vert_{x_3 =0} - \kappa_B \partial_3 \theta_B\vert_{x_3 =0} ) \phi_S { \ }d x_h = \int_{\mathbb{R}^2} q_S \phi_S { \ }d x_h.
\end{equation*}
Since the above equation holds for all $\phi_S \in C_0^\infty (\mathbb{R}^2)$, we see that $q_S = \kappa_S \Delta_h \theta_S + \kappa_A \partial_3 \theta_A \vert_{x_3=0} - \kappa_B \partial_3 \theta_B \vert_{x_3=0}$. Therefore, Proposition \ref{prop62} is proved.
\end{proof}

From Proposition \ref{prop62} we set
\begin{equation}\label{eq64}
\begin{cases}
Q_A = Q_A(x,t) = \kappa_A \Delta \theta_A,\\
Q_B = Q_B (x,t) = \kappa_B \Delta \theta_B,\\
Q_S = Q_S ( x_h , t) = \kappa_S \Delta_h \theta_S + \kappa_A \partial_3 \theta_A \vert_{x_3 =0} - \kappa_B \partial_3 \theta_B \vert_{x_3 =0}.
\end{cases}
\end{equation}

Now we assume that the time rate of change of the heat energies equals to the forces derived from the variation of energies dissipation due to thermal diffusions, that is, suppose that for every $0<t<T$ and $\Lambda \subset \mathbb{R}^3$,
\begin{align*}
\frac{d}{d t}\int_{\mathbb{R}_{+}^3 \cap \Lambda} \rho_A \mathcal{C}_A \theta_A { \ }d x = \int_{\mathbb{R}_{+}^3 \cap \Lambda} Q_A { \ }d x,\\
\frac{d}{d t}\int_{\mathbb{R}_{-}^3 \cap \Lambda} \rho_B \mathcal{C}_B \theta_B { \ }d x = \int_{\mathbb{R}_{-}^3 \cap \Lambda} Q_B { \ }d x,\\
\frac{d}{d t}\int_{\mathbb{R}^2 \cap \Lambda} \rho_S \mathcal{C}_S \theta_S { \ }d x_h = \int_{\mathbb{R}^2 \cap \Lambda} Q_S { \ }d x_h.
\end{align*}
Then we apply the transport theorems to derive
\begin{equation}\label{eq65}
\begin{cases}
\rho_A \partial_t ( \mathcal{C}_A \theta_A) + \rho_A (v_A \cdot \nabla) (\mathcal{C}_A \theta_A) = Q_A & \text{ in } \mathbb{R}^3_{+} \times (0, T),\\
\rho_B \partial_t ( \mathcal{C}_B \theta_B ) + \rho_B (v_B \cdot \nabla) (\mathcal{C}_B \theta_B) = Q_B & \text{ in } \mathbb{R}^3_{-} \times (0, T),\\
\rho_S \partial_t ( \mathcal{C}_S \theta_S) + \rho_S (v_S \cdot \nabla_h) ( \mathcal{C}_S \theta_S) = Q_S & \text{ in } \mathbb{R}^2 \times (0, T).
\end{cases}
\end{equation}
We assume that $v_A \equiv { }^t (0,0,0)$, $v_B \equiv { }^t (0,0,0)$, $v_S \equiv { }^t (0,0)$, $\mathcal{C}_\sharp$, $\rho_\sharp$ are positive constants, and that $\rho_\sharp \mathcal{C}_\sharp \equiv \alpha_\sharp$ for some $\alpha_\sharp \in \mathbb{R}_+$. Combining \eqref{eq64}, \eqref{eq65}, and \eqref{eq62}, we have
\begin{equation*}
\begin{cases}
\alpha_A \partial_t \theta_A = \kappa_A \Delta \theta_A & \text{ in } \mathbb{R}^3_{+} \times (0, T),\\
\alpha_B \partial_t \theta_B = \kappa_B \Delta \theta_B & \text{ in } \mathbb{R}^3_{-} \times (0, T),\\
\alpha_S \partial_t \theta_S = \kappa_S \Delta_h \theta_S + \kappa_A \partial_3 \theta_A \vert_{x_3=0} - \kappa_B \partial_3 \theta_B \vert_{x_3 =0} & \text{ in } \mathbb{R}^2 \times (0, T),\\
\theta_A \vert_{x_3 =0} = \theta_B \vert_{x_3 =0} = \theta_S & \text{ in }\mathbb{R}^2 \times (0, T).
\end{cases}
\end{equation*}
Therefore, we have our heat equations \eqref{eq11}.

\section{Appendix (II): H\"{o}lder Continuity}\label{sect7}

In Appendix (II), we derive \eqref{eq23} in Lemma \ref{lem23}. Let $0 \leq q \leq 1$. Since ${\rm{e}}^{- t \mathcal{L}}$ is a bounded analytic semigroup on $\mathcal{H}$, we see that there is $C = C (q,T) >0$ such that for all $\phi_1 \in \mathcal{H}$, $\phi_2 \in D (\mathcal{L}^q)$, and $0 < t < T$,
\begin{align}
\Vert \mathcal{L}^q {\rm{e}}^{- t \mathcal{L}} \phi_1 \Vert_{\mathcal{H}} & \leq C t^{-q} \Vert \phi_1 \Vert_{\mathcal{H}},\label{eq71}\\
\Vert ({\rm{e}}^{- t \mathcal{L}} -  1) \phi_2 \Vert_{\mathcal{H}} & \leq C t^q \Vert \mathcal{L}^q \phi_2 \Vert_{\mathcal{H}}.\label{eq72}
\end{align}
See \cite[Theorem 6.13 in Chapter II]{Paz83} for the derivations of \eqref{eq71} and \eqref{eq72}.
\begin{proof}[Proof of Lemma \ref{lem23}]

Let $T \in (0,\infty)$, $0 < \eta \leq 1/2$, $V_0 \in D (\mathcal{L}^{1/2})$, and $F \in C^{\eta}_{loc}((0,T); \mathcal{H}) \cap L^2(0,T; \mathcal{H})$. Fix $\varepsilon,T_0 >0$ such that $\varepsilon <T_0 < T$. Let $t_1,t_2 >0$ such that $\varepsilon \leq t_1 \leq t_2 \leq T_0$.
From
\begin{align*}
& V (t_2) = {\rm{e}}^{- t_2 \mathcal{L}} V_0 + \int_0^{t_2} {\rm{e}}^{- ( t_2 - \tau ) \mathcal{L} } \mathcal{F} (\tau ) { \ }d \tau,\\
& V (t_1) = {\rm{e}}^{- t_1 \mathcal{L}} V_0 + \int_0^{t_1} {\rm{e}}^{- ( t_1 - \tau ) \mathcal{L} } \mathcal{F} (\tau ) { \ }d \tau,
\end{align*}
we have
\begin{multline*}
V (t_2) - V (t_1) = \{ {\rm{e}}^{- ( t_2 - t_1 ) \mathcal{L}} - 1 \} {\rm{e}}^{- t_1 \mathcal{L}}V_0\\
+ \int_{t_1}^{t_2} {\rm{e}}^{- ( t_2 - \tau ) \mathcal{L}} \mathcal{F} ( \tau ) { \ }d \tau + \int_0^{t_1} \{ {\rm{e}}^{- ( t_2 - t_1 ) \mathcal{L}} - 1 \}{\rm{e}}^{-( t_1 - \tau ) \mathcal{L}} \mathcal{F} (\tau ) { \ } d \tau
\end{multline*}
and
\begin{equation*}
\mathcal{L} V (t_2) - \mathcal{L} V (t_1) = P_1 (t_2,t_1) + P_2 (t_2,t_1) + P_3 (t_2,t_1) + P_4 (t_2,t_1) + P_5 (t_2,t_1).
\end{equation*}
Here
\begin{align*}
P_1 = P_1 (t_1,t_2) &:= ( {\rm{e}}^{- (t_2 -t_1 ) \mathcal{L} } - 1 ) \mathcal{L} {\rm{e}}^{- t_1 \mathcal{L}} V_0,\\
P_2 = P_2 (t_1 , t_2 ) &:= \mathcal{L} \int_{t_1}^{t_2} {\rm{e}}^{- (t_2 - \tau ) \mathcal{L}}\{ \mathcal{F} (\tau ) - \mathcal{F} (t_2) \} { \ }d \tau,\\
P_3 = P_3 (t_1 , t_2 ) &:= \mathcal{L} \int_0^{t_1} \{ {\rm{e}}^{- ( t_2 - t_1 ) \mathcal{L}} - 1 \} {\rm{e}}^{- ( t_1 - \tau ) \mathcal{L}} \{ \mathcal{F} (\tau ) - \mathcal{F} (t_1) \} { \ } d \tau ,\\
P_4 = P_4 (t_1 , t_2) &:= \mathcal{L} \int_{t_1}^{t_2} {\rm{e}}^{- (t_2 - \tau ) \mathcal{L}} \mathcal{F} (t_2) { \ }d \tau,\\
P_5 = P_5 (t_1 , t_2) &:= \mathcal{L} \int_0^{t_1} {\rm{e}}^{- (t_1 - \tau ) \mathcal{L}} \{ {\rm{e}}^{- (t_2 - t_1 ) \mathcal{L}} - 1  \} \mathcal{F} (t_1) { \ }d \tau.
\end{align*}
By \eqref{eq72}, we see that
\begin{align*}
\Vert P_1 \Vert_{\mathcal{H}} & \leq C (t_2 - t_1 )^{1/2} \Vert \mathcal{L} {\rm{e}}^{- t_1 \mathcal{L}} \mathcal{L}^{1/2} V_0 \Vert_{\mathcal{H}}\\
& \leq \frac{C(t_2 - t_1)^{1/2}}{\varepsilon} \Vert \mathcal{L}^{1/2} V_0 \Vert_{\mathcal{H}}.
\end{align*}
Since $\mathcal{F} \in C^{\eta}_{loc}((0,T); \mathcal{H})$, we apply \eqref{eq71} to check that
\begin{equation*}
\Vert P_2 \Vert_{\mathcal{H}} \leq \int_{t_1}^{t_2} C \frac{(t_2 - \tau )^\eta}{(t_2 - \tau )} { \ }d \tau \leq C (t_2 - t_1)^{\eta}
\end{equation*}
and that
\begin{align*}
\Vert P_3 \Vert_{\mathcal{H}} & \leq C (t_2 - t_1)^{\eta/2} C \left\Vert \mathcal{L}^{1+ \eta/2} {\rm{e}}^{- (t_1 - \tau ) \mathcal{L}} \{ \mathcal{F} (\tau ) - \mathcal{F} (t_1) \}{ \ }d \tau \right\Vert_{\mathcal{H}}\\
& \leq C (t_2 - t_1)^{\eta/2} \int_0^{t_1} \frac{(t_1 - \tau )^\eta}{( t_1 - \tau )^{1+ \eta/2}} { \ }d \tau\\
& \leq C (t_2 - t_1 )^{\eta/2} t_1^{\eta/2} \leq C (t_2 - t_1)^{\eta/2} T^{\eta/2}.
\end{align*}
From
\begin{equation*}
\frac{d}{d \tau } {\rm{e}}^{- \tau \mathcal{L}} \phi = - \mathcal{L} {\rm{e}}^{- \tau \mathcal{L}} \phi { \ \ \ }( \phi \in \mathcal{H}),
\end{equation*}
we find that
\begin{equation*}
P_4 + P_5 = F (t_2) - F (t_1) - {\rm{e}}^{- (t_2 - t_1) \mathcal{L}} \{ F(t_2) - F (t_1) \} - ({\rm{e}}^{- (t_2 - t_1) \mathcal{L}} - 1){\rm{e}}^{- t_1 \mathcal{L}} F (t_1).
\end{equation*}
By \eqref{eq71} and \eqref{eq72}, we see that
\begin{align*}
\Vert P_4 + P_5 \Vert_{\mathcal{H}} & \leq C (t_2 - t_1)^{\eta} + C (t_2 - t_1)^{\eta} \Vert \mathcal{L}^{\eta} {\rm{e}}^{- t_1 \mathcal{L}} \mathcal{F} (t_1) \Vert_{\mathcal{H}}\\
& \leq C (t_2 - t_1)^{\eta/2} + \frac{C (t_2 - t_1)^{\eta}}{ \varepsilon^\eta} \sup_{\varepsilon \leq \tau \leq T_0} \Vert \mathcal{F} ( \tau ) \Vert_{\mathcal{H}}.   
\end{align*}
As a result, we have
\begin{equation*}
\Vert \mathcal{L} V (t_2) - \mathcal{L} V (t_1) \Vert_{\mathcal{H}} \leq C (t_2 - t_1 )^{\eta/2}. 
\end{equation*}
Therefore, we see that for each fixed $\varepsilon,T_0 (0 < \varepsilon < T_0 < T)$
\begin{equation*}
\mathcal{L} V \in C^{\eta/2}([\varepsilon , T_0] ; \mathcal{H}).
\end{equation*}
This shows that
\begin{equation}\label{eq73}
\mathcal{L} V \in C_{loc}^{\eta/2}( (0, T); \mathcal{H}).
\end{equation}
Using an argument similar to deduce \eqref{eq73} with \eqref{eq71}, \eqref{eq72} and\\ $\mathcal{F} \in L^2(0,T;\mathcal{H})$, we find that
\begin{equation*}
V \in C_{loc}^{1/4}( (0, T); \mathcal{H}) \subset C_{loc}^{\eta/2}( (0, T); \mathcal{H}).
\end{equation*}
Since $dV/{dt} = - \mathcal{L} V + \mathcal{F}$ and $\mathcal{L} V , \mathcal{F} \in C_{loc}^{\eta/2} ((0,T); \mathcal{H})$, we check that
\begin{equation*}
dV/{dt} \in C_{loc}^{\eta/2}( (0, T); \mathcal{H}).
\end{equation*}
Therefore, we see \eqref{eq23}.
  \end{proof}

\end{document}